\documentclass[12pt]{amsart}
\usepackage{amsmath}
\usepackage{amsthm}
\usepackage{amssymb}
\usepackage{a4wide}

\usepackage{pdfsync}

\usepackage{hyperref}

\theoremstyle{plain}

\newtheorem{cor}{Corrolary}[section]
\newtheorem{lemma}{Lemma}[section]
\newtheorem{theorem}{Theorem}[section]
\newtheorem{remark}{Remark}

\def\zz{\mathbb{Z}}
\def\rr{\mathbb{R}}
\def\nn{\mathbb{N}}
\def\cn{\mathbb{C}}
\def\eps{\epsilon}
\renewcommand{\Im}{\,{\rm Im\, }}
\renewcommand{\Re}{\,{\rm Re\, }}

\title{Dispersive properties for discrete Schr\"odinger equations}
\author{Liviu I. Ignat, Diana Stan}

\address{ L. I. Ignat
\hfill\break\indent Institute of Mathematics ``Simion Stoilow'' of the Romanian Academy,
\hfill\break\indent P.O. Box 1-764, RO-014700 Bucharest, Romania
\hfill\break\indent \and
\hfill\break\indent BCAM - Basque Center for Applied Mathematics,
 \hfill\break\indent Bizkaia Technology Park, Building 500 Derio, Basque Country, Spain
}
 \email{{\tt
liviu.ignat@gmail.com}\hfill\break\indent  {\it Web page: }{\tt
http://www.imar.ro/\~\,lignat}}

\address{D. Stan
\hfill\break\indent Institute of Mathematics ``Simion Stoilow'' of the Romanian Academy,
\hfill\break\indent P.O. Box 1-764, RO-014700 Bucharest, Romania
}
  \email{{\tt dianastan2008@gmail.com }
}

\begin{document}

	\begin{abstract}
	In this paper we prove dispersive estimates for the system formed by two coupled  discrete Schr\"odinger equations. We obtain estimates for the resolvent of the discrete operator and prove that it satisfies the limiting absorption principle. The decay of the solutions is proved by using classical and some new results on oscillatory
	integrals.
	\end{abstract}


\maketitle

\section{Introduction}
\setcounter{equation}{0}

Let us consider the linear Schr\"{o}dinger equation  (LSE):
\begin{equation}\label{sch1}
\left\{\begin{array}{l}
  iu_t+ u_{xx} = 0,  \,x\in \rr,\,t\neq 0,\\
  u(0,x)  =\varphi(x), \,x\in \rr.
    \end{array}\right.
\end{equation}

Linear equation \eqref{sch1} is solved by $u(t,x)=S(t)\varphi$,
where $S(t)=e^{it\Delta}$  is the free Schr\"{o}dinger operator. The
linear semigroup has two important properties. First, the
conservation of the $L^2$-norm:
\begin{equation}\label{energy}
\|S(t)\varphi\|_{L^2(\rr)}=\|\varphi \|_{L^2(\rr)}
\end{equation}
 and a dispersive estimate of the form:
\begin{equation}\label{linfty}
|(S(t)\varphi)(x)|\leq \frac 1{(4\pi |t|)^{1/2}}\|\varphi \|_{L^1(\rr)}, \
x\in \rr,\ t\neq 0.
\end{equation}


The space-time estimate
\begin{equation}\label{l6l6}
    \|S(\cdot)\varphi\|_{L^{6}(\rr,\,L^{6}(\rr))}\leq C\|\varphi\|_{L^2(\rr)},
\end{equation}
 due to  Strichartz \cite{0372.35001}, is deeper. It guarantees  that the
solutions of system \eqref{sch1} decay  as $t$ becomes large and that they gain some
spatial integrability.
Inequality (\ref{l6l6}) was generalized by Ginibre and Velo
\cite{MR801582}. They proved  the mixed space-time estimates, well
known as Strichartz estimates:
\begin{equation}\label{dsch0001}
    \|S(\cdot)\varphi\|_{L^q(\rr,\,L^r(\rr))}\leq C(q,r)\|\varphi\|_{L^2(\rr)}
\end{equation}
 for  the so-called \textit{admissible pairs} $(q,r)$:
 \begin{equation}\label{adm}
\frac 1 {q}=\frac 12\big(\frac 12-\frac 1r\big), \quad 2\leq q,r\leq \infty.
\end{equation}
Similar results can be stated in any space dimension but it is beyond  the scope of this article. These estimates have been successfully applied to obtain well-posedness results  for the nonlinear Schr\"odinger equation (see  \cite{1055.35003}, \cite{MR2233925} and the reference therein).

Let us now consider the following system of difference equations
\begin{equation}\label{dsch}
\left\{
\begin{array}{ll}
i u_t + \Delta_d u=0,&j\in \zz, t\neq 0,\\[10pt]
u(0)=\varphi,&
\end{array}
\right.
\end{equation}
where $\Delta_d$ is the discrete laplacian defined by
 $$(\Delta_d u)(j)=u_{j+1}-2u_j+u_{j-1},\quad j\in \zz.$$
Concerning the long time behavior of the solutions of system \eqref{dsch} in \cite{MR2150357}
 the authors have proved that a  decay property  similar to the one obtained for the continuous Schr\"odinger equation holds:
\begin{equation}\label{decay.discrete}
\|u(t)\|_{l^\infty(\zz)}\leq C(|t|+1)^{-1/3}\|\varphi\|_{l^1(\zz)},\quad \forall \ t\neq 0.
\end{equation}
The proof of \eqref{decay.discrete} consists in writing the solution $u$ of \eqref{dsch} as the convolution between a kernel $K_t$ and the initial data $\varphi$ and then estimate $K_t$ by using  Van der Corput's lemma.
For the linear semigroup $\exp(it\Delta_d)$,
Strichartz like estimates similar to those in \eqref{dsch0001} have been obtained in \cite{MR2150357}
 for a larger class of pairs $(q,r)$:
\begin{equation}\label{pairs.discrete}
\frac{1}q\leq \frac 13\big(\frac 12-\frac 1r\big),\quad 2\leq q,r\leq \infty.
\end{equation}

We also mention \cite{1063.35016} and \cite{MR2485456} where the authors consider a similar equation on $h\zz$ by replacing $\Delta_d$ by $\Delta_d/h^2$ and analyze the same properties in the context of numerical approximations of the linear and nonlinear Schr\"odinger equation.

A more thorough analysis has been done in \cite{MR2282998} and  \cite{MR2468536} where the authors analyze the  decay properties of the solutions of  equation $iu_t+Au=0$ where $A=\Delta_d-V$,  with $V$ a real-valued potential. In these papers $l^1(\zz)-l^\infty(\zz)$ and $l^2_{-\sigma}(\zz)-l^2_{\sigma}(\zz)$ estimates for $\exp(itA)P_{a,c}(A)$ have been obtained where $P_{a,c}(A)$ is the spectral projection to the absolutely continuous spectrum of $A$ and $l^2_{\pm \sigma}(\zz)$ are weighted $l^2(\zz)$-spaces.

In what concerns the Sch\"odinger equation with variable coefficients we mention the results of Banica
\cite{MR2049025}. Consider a partition of the real axis as follows:
$-\infty=x_0<x_1<\dots<x_{n+1}=\infty$
and a step function
$\sigma(x)=b_i^{-2} \ \text{for}\ x\in (x_i,x_{i+1}),$
where $b_i$ are positive numbers.
The solution $u$ of the Schr\"odinger equation
$$
\left\{
\begin{array}{ll}
iu_t(t,x)+(\sigma(x)u_x)_x(t,x)=0,& \text{for}\ x\in \rr,t\neq 0,\\[10pt]
u(0,x)=u_0(x), &x\in \rr,
\end{array}
\right.$$
satisfies the dispersion inequality
$$\|u(t)\|_{L^\infty(\rr)}\leq C|t|^{-1/2}\|u_0\|_{L^1(\rr)}, \quad \ t\neq 0,$$
where constant $C$ depends on $n$ and on sequence $\{b_i\}_{i=0}^n$. We recall that in \cite{liv-tree} the above result was used in  the analysis of the long time behavior of the solutions of the linear Sch\"odinger equation on  regular trees.  In the case of discrete equations the corresponding model is given by
\begin{equation}\label{system.A}
\left\{
\begin{array}{ll}
iU_t+AU=0,\ t\neq 0,\\
U(0)=\varphi,
\end{array}
\right.
\end{equation}
where  the infinite matrix $A$ is symmetric with a finite number of diagonals nonidentically vanishing.
Once a result  similar to \cite{MR2049025} will be obtained for discrete Schr\"odinger equations with non-constant coefficients we can apply it to obtain dispersive estimates for discrete Schr\"odinger equations on trees. But as far as we know the study of the decay properties of solutions of system \eqref{system.A} in terms of the properties of $A$
 is a difficult task and we try to give here a partial answer to this problem. In the case when $A$ is a diagonal matrix these properties are easily obtained by using the Fourier transform and classical estimates for oscillatory integrals.


The main goal of this article is to analyze  a simplified model which consists in coupling two DSE by Kirchhoff's type condition:
\begin{equation}\label{model1}
\left\{
\begin{array}{ll}
  i u_t(t,j) + b_1^{-2}(\Delta_d u)(t,j) = 0& j\leq -1,\ t\neq 0, \\[5pt]
i v_t(t,j) + b_2^{-2}(\Delta_d v)(t,j) = 0& j\geq 1, \ t\neq 0, \\[10pt]
  u(t,0) = v(t,0) , &t \neq 0, \\[5pt]
  b_1^{-2}(u(t,-1)-u(t,0))=b_2^{-2}(v(t,0)-v(t,1)),&t\neq 0,\\[10pt]
  u(0,j) = \varphi(j),& j \leq -1,\\[5pt]
    v(0,j) = \varphi(j),& j \geq 1.
\end{array}
\right.
\end{equation}
In the above system $u(t,0)$ and $v(t,0)$ have been artificially introduced to couple the two equations on positive and negative integers.
The third condition in the above system requires continuity along the interface $j=0$ and the fourth one can be interpreted as the continuity of the flux along the interface.

The main result of this paper is given in the following theorem.

\begin{theorem}\label{main}
For any $\varphi\in l^2(\zz\setminus\{0\})$ there exists a unique solution $(u,v)\in C(\rr,l^2(\zz\setminus\{0\}))$ of system \eqref{model1}. Moreover, there exists a positive constant $C(b_1,b_2)$ such that
\begin{equation}\label{decay.main}
\|(u,v)(t)\|_{l^\infty(\zz\setminus\{0\})}\leq C(b_1,b_2)(|t|+1)^{-1/3}\|\varphi\|_{l^1(\zz\setminus\{0\})},\quad \forall t\in \rr,
\end{equation}
holds for all $\varphi\in l^1(\zz\setminus\{0\})$.
\end{theorem}

Using the well-known results of Keel and Tao \cite{0922.35028} we obtain the following Strichartz-like estimates for the solutions of system \eqref{model1}.

\begin{theorem}\label{stric}For any $\varphi\in l^2(\zz\setminus\{0\})$ the solution $(u,v)$ of system \eqref{model1} satisfies
$$\|(u,v)\|_{L^q(\rr,\, l^r(\zz\setminus\{0\}))}\leq C(q,r)\|\varphi\|_{l^2(\zz\setminus\{0\})}$$
for all pairs  $(q,r)$ satisfying \eqref{pairs.discrete}.
\end{theorem}

The paper is organized as follows: In section \ref{simplified} we present some discrete models,
in particular system \eqref{model1}  in  the  case $b_1=b_2$
  and show how it is related with problem \eqref{dsch}.
 In addition, a system with a dynamic coupling  along the interface is presented.
  In section \ref{integrale.oscilatorii} we present some classical  results on oscillatory integrals and make some improvements that we will need in the proof of Theorem \ref{main}. In section \ref{main.section} we obtain an explicit formula for the resolvent associated with system \eqref{model1}. We prove a limiting absorption principle  and we give the proof of the main result of this paper.
Finally we present some open problems.

\section{Some discrete models}\label{simplified}
\setcounter{equation}{0}

In this section in order to emphasize the main differences and difficulties with respect to the continuous case  when we deal with discrete systems we will consider two models. In the first case
we consider system \eqref{model1}  with the two coefficients in the front of the discrete laplacian equal. In the following we denote $\zz^*=\zz\setminus\{0\}.$

\begin{theorem}\label{main.particular}
Let us assume that $b_1=b_2$. For any $\varphi \in l^2(\zz^*)$ there exists a unique solution $u\in C(\rr,l^2(\zz^*))$ of system \eqref{model1}. Moreover there exists a positive constant $C(b_1)$ such that
\begin{equation}\label{disp.1}
\|u(t)\|_{l^\infty(\zz^*)}\leq C(b_1)(|t|+1)^{-1/3}\|\varphi\|_{l^1(\zz^*)},\quad \forall \ t\in \rr,
\end{equation}
holds for all $\varphi\in l^1(\zz^*)$.
\end{theorem}

In the particular case considered here  we can reduce the proof of the dispersive estimate \eqref{disp.1} to the analysis of two problems: one with Dirichlet's boundary condition and another one with a discrete Neumann's boundary condition.

Before starting the proof of Theorem \ref{main.particular} let us recall that in the case of system \eqref{dsch}
its solution is given by $u(t)=K_t\ast \varphi$ where $\ast$ is the standard convolution on $\zz$ and
$$K_t(j)=\int_{-\pi}^{\pi}e^{-4it\sin^2(\frac \xi2)}e^{ij\xi}d\xi,\quad t\in \rr,\ j\in \zz.$$

In \cite{MR2150357}
a simple argument based on Van der Corput's lemma has been used to show that for any real number $t$ the following holds:
\begin{equation}\label{decay.k}
|K_t(j)|\leq C(|t|+1)^{-1/3},\quad  \forall j\in \zz.
\end{equation}

\begin{proof}[Proof of Theorem \ref{main.particular}]
The existence of the solutions is immediate since operator $A$  defined in \eqref{matrix.a}  is bounded in $l^2(\zz^*)$. We prove now the decay property \eqref{main.particular}. Let us restrict  for simplicity to the case $b_1=b_2=1$.

For $(u,v)$ solution of  system \eqref{model1}
let us set
$$S(j)=\frac{v(j)+u(-j)}2,\quad D(j)=\frac{v(j)-u(-j)}2,j\geq 0.$$
Observe that $u$ and $v$ can be recovered from $S$ and $D$ as follows
$$(u,v)=((S-D)(-\cdot),S+D).$$
Writing the equations satisfied by $u$ and $v$  we obtain that $D$ and $S$  solve two discrete Schr\"odinger equations on $Z^+=\{j\in \zz, j\geq 1\}$ with Dirichlet, respectively Neumann boundary conditions:
\begin{equation}\label{system.D}
\left\{
\begin{array}{ll}
  i D_t(t,j) + (\Delta_d D)(t,j) = 0&j\geq 1,\ t\neq 0, \\[5pt]
  D(t,0) = 0, & t\neq 0, \\[5pt]
  D(0,j) = \frac{\varphi(j)-\varphi(-j)}2, &j\geq 1,
\end{array}
\right.
\end{equation}
and
\begin{equation}\label{system.S}
\left\{
\begin{array}{ll}
  i S_t(t,j) + (\Delta_d S)(t,j) = 0& j\geq 1,\ t\neq 0, \\[5pt]
  S(t,0) = S(t,1) , &t\neq 0, \\[5pt]
  S(0,j) = \frac{\varphi(j)+\varphi(-j)}2,& j \geq 1.
\end{array}
\right.
\end{equation}

Making an odd extension of the function $D$ and using the representation formula for the solutions of \eqref{dsch} we obtain that the solution of the
 Dirichlet problem
\eqref{system.D} satisfies
\begin{equation}\label{formula.d}
D(t,j)=\sum_{k \geq 1}(K_t(j-k) - K_t(j+k))D(0,k), \quad t\neq 0,\  j\geq 1.
\end{equation}
A similar even extension of  function $S$ permits us to obtain the explicit formula for the solution of the Neumann problem \eqref{system.S}
\begin{equation}\label{formula.n}
S(t,j)=\sum_{k\geq 1}(K_t(k-j) +K_t(k+j-1) ) S(0,k),\quad t\neq 0,\ j\geq 1.
\end{equation}
Using the decay of the kernel $K_t$ given by \eqref{decay.k} we obtain that $S(t)$ and $D(t)$ decay as $(|t|+1)^{-1/3}$ and then  the same property holds for $u$ and $v$. This finishes the proof of this particular case.
\end{proof}

Observe that our proof has taken into account the particular structure of the equations. When the coefficients $b_1$ and $b_2$ are not equal we cannot write an equation verified by
 functions $D$ or $S$.

We now write system \eqref{model1} in matrix formulation.
Using the coupling conditions at $j=0$
 system \eqref{model1} can be written in the following equivalent form
$$\left\{
\begin{array}{ll}
iU_t+AU=0,\\
U(0)=\varphi,
\end{array}
\right.$$
where $U=(u,v)^T$, $u=(u(j))_{j\leq -1}$, $v=(v_j)_{j\geq 1}$ and
\begin{equation}\label{matrix.a}
A=\left(\begin{array}{cccccccc}
... & ... & ... & 0 & 0 & 0 & 0 & 0 \\
0 & b_1^{-2} & -2b_1^{-2} & b_1^{-2} & 0 & 0 & 0 & 0 \\
0 & 0 & b_1^{-2} & - b_1^{-2}-\frac{1}{b_1^2+b_2^2}& \frac{1}{b_1^2+b_2^2}& 0 & 0 & 0 \\
0 & 0 & 0 & \frac{1}{b_1^2+b_2^2}& -\frac{1}{b_1^2+b_2^2}-b_2^{-2} & b_2^{-2} & 0 & 0 \\
0 & 0 & 0 & 0 & b_2^{-2} & -2b_2^{-2} & b_2^{-2} & 0 \\
0 & 0 & 0 & 0 & 0 & ... & ... & ...
\end{array}\right).
\end{equation}
In the particular case $b_1=b_2=1$ the operator $A$ can be decomposed as follows
$$A=\Delta_d+B=\left(\begin{array}{cccccccc}
... & ... & ... & 0 & 0 & 0 & 0 & 0 \\
0 & 1 & -2 & 1 & 0 & 0 & 0 & 0 \\
0 & 0 & 1 & -2 &  1 & 0 & 0 & 0 \\
0 & 0 & 0 & 1 & -2& 1 & 0 & 0 \\
0 & 0 & 0 & 0 & 1 & -2 & 1 & 0 \\
0 & 0 & 0 & 0 & 0 & ... & ... & ...\end{array}\right)+\left(\begin{array}{cccccccc}... & ... & ... & 0 & 0 & 0 & 0 & 0 \\
0 & 0 & 0 & 0 & 0 & 0 & 0 & 0 \\
0 & 0 & 0 & \frac 12 & -\frac 12 & 0 & 0 & 0 \\
0 & 0 & 0 & -\frac 12 & \frac 12 & 0 & 0 & 0 \\
0 & 0 & 0 & 0 & 0 & 0 & 0 & 0 \\
0 & 0 & 0 & 0 & 0 & ... & ... & ...
\end{array}\right).$$
However, we do not know how to use the dispersive properties of $\exp(it\Delta_d)$ and the particular structure of $B$ in order to obtain the decay of the new semigroup $\exp(it(\Delta_d+B))$.

Another model of interest is the following one inspired in the numerical approximations of LSE.
Set
$$a(x)=\left\{
\begin{array}{ll}
b_1^{-2},&x<0,\\
b_2^{-2},&x>0.
\end{array}
\right.$$
Using the following discrete derivative operator
$$(\partial u)(x)=u(x+\frac 12)-u(x-\frac 12)$$
we can introduce the second order discrete operator
\begin{align*}
\partial (a\partial u)(j)=a(j+\frac 12)u(j+1)-\big(a(j+\frac 12)+a(j-\frac 12)\big)u(j)+a(j-\frac 12)u(j-1),j\in \zz.
\end{align*}

In this case we have to analyze the following system
%
\begin{equation}\label{model2}
\left\{
\begin{array}{ll}
  i u_t(t,j) + b_1^{-2}(\Delta_d u)(t,j) = 0,& j\leq -1,\ t\neq 0, \\[5pt]
i u_t(t,j) + b_2^{-2}(\Delta_d u)(t,j) = 0,& j\geq 1, \ t\neq 0, \\[5pt]
iu_t(t,0)+b_1^{-2}u(t,-1)-(b_1^{-2}+b_2^{-2})u(t,0)+b_2^{-1}u(t,1)=0,&t\neq 0,\\[5pt]
  u(0,j) = \varphi(j),& j \in\zz.\\
\end{array}
\right.
\end{equation}
In matrix formulation it reads $iU_t+AU=0$ where $U=(u(j))_{j\in \zz}$, and the operator $A$ is given by the following one
\begin{equation}\label{matrice.model2}
A=\left(\begin{array}{ccccccc}... & ... & ... & 0 & 0 & 0 & 0 \\
0 & b_1^{-2} & -2b_1^{-2} & b_1^{-2} & 0 & 0 & 0 \\
0 & 0 & b_1^{-2} & -(b_1^{-2}+b_2^{-2}) & b_2^{-2} & 0 & 0 \\
0 & 0 & 0 & b_2^{-2} & -2b_2^{-2} & b_2^{-2} & 0 \\
0 & 0 & 0 & 0 & ... & ... & ...\end{array}\right).
\end{equation}
Observe that in the case $b_1=b_2$ the results of
 \cite{MR2150357} give us the decay of the solutions.

Regarding the long time behavior of the solutions of system \eqref{model2} we have the following result.
\begin{theorem}\label{main.model.2}
For any $\varphi\in l^2(\zz)$ there exists a unique solution $u\in C(\rr,l^2(\zz))$ of system \eqref{model2}. Moreover, there exists a positive constant $C(b_1,b_2)$ such that
$$\|u(t)\|_{l^\infty(\zz)}\leq C(b_1,b_2)(|t|+1)^{-1/3}\|\varphi\|_{l^1(\zz)},\quad \forall t\in \rr,$$
holds for all $\varphi\in l^1(\zz)$.
\end{theorem}
The proof of this result is similar to the one of Theorem \ref{main} and we will only sketch it at the end of Section
\ref{main.section}.

\section{Oscillatory integrals}\label{integrale.oscilatorii}
\setcounter{equation}{0}

In this section we present some classical tools for oscillatory integrals and we give an improvement of Van der Corput's Lemma that is in some sense similar to the one obtained in \cite{0738.35022}. First of all let us recall  Van der Corput's lemma(see for example \cite{0821.42001}, p. 332).
\begin{lemma}(Van der Corput)\label{vandercorput}
Let $k \geq 1$ be an integer, and $\phi: [a,b] \rightarrow \rr$ such that $|\phi^{(k)}(x)|\geq 1$ for all $x \in [a,b]$, and $\phi'$ monotone in the case $k=1$.\\
Then $$\left| \int_a ^b e^{it\phi(x)} \psi (x) dx \right| \leq c_k |t| ^{-\frac{1}{k}}\big(\| \psi \|_{L^{\infty}(a,b)}+\int_a ^b |\psi '(\xi) | d\xi \big) ,\quad
\forall\ t\neq 0.$$
\end{lemma}

\medskip

A first improvement has been obtained in \cite{0738.35022} where the authors analyze the smoothing effect of some dispersive equations.
We will present here a particular case of the results  in \cite{0738.35022}, that will be sufficient for our purposes. In the sequel $\Omega$ will be a bounded interval.
We consider class $\mathcal{A}_2 $ of real functions $\phi\in C^3(\overline \Omega)$ satisfying the following conditions:\\
1) Set $S_\phi=\{\xi\in \Omega: \phi''=0\}$ is finite,\\
2) If $\xi_0\in S_\phi$ then there exist constants $\eps, c_1, c_2$ and $\alpha\geq 2$ such that for all $|\xi-\xi_0|<\eps$,
$$c_1|\xi-\xi_0|^{\alpha-2}\leq |\phi''(\xi)|\leq c_2|\xi-\xi_0|^{\alpha-2},$$
3) $\phi''$ has a finite number of changes of monotonicity.
 \begin{lemma}\label{kpv}
 Let $\Omega$ be a bounded interval, $\phi\in \mathcal{A}_2$ and
 $$I(x,t)=\int _\Omega e^{i(t\phi(\xi)-x\xi)}|\phi''(\xi)|^{1/2}d\xi.$$
 Then for any $x,t\in \rr$
 \begin{equation}\label{est.kpv}
|I(x,t)|\leq c_{\phi} |t|^{-1/2},
\end{equation}
where $c_\phi$ depends only on the constants involved in the definition of class $\mathcal{A}_2$.
 \end{lemma}

\begin{remark}
The results of \cite{0738.35022} are more general that the one we presented here allowing functions with vertical asymptotics, finite union of intervals or infinite domains.
\end{remark}

As a corollary we also have \cite{0738.35022}:
\begin{cor}
If $\phi \in \mathcal{A}_2$ then
$$\Big|\int _\Omega e^{i(t\phi(\xi)-x\xi)}|\phi''(\xi)|^{1/2}\psi(\xi)d\xi \Big| \leq C_{\phi}|t|^{-1/2}\Big( \|\psi\|_{L^{\infty}(\Omega)} + \int_{\Omega}|\phi ' (\xi)|d\xi \Big),$$
holds for all  $x,t\in \rr$.
\end{cor}

In the proof of our main result we will need a result similar to Lemma \ref{kpv} but with $|p'''|^{1/3}$ instead of $|p{''}|^{1/2}$ in the definition of $I(x,t)$. We define class $\mathcal{A}_3 $ of real functions $\phi\in C^4(\overline{\Omega})$ satisfying the following conditions:\\
1) Set $S_\phi=\{\xi\in \Omega: \phi'''=0\}$ is finite,\\
2) If $\xi_0\in S_\phi$ then there exist constants $\eps, c_1, c_2$ and $\alpha\geq 3$ such that for all $|\xi-\xi_0|<\eps$,
\begin{equation}\label{expansion.3d}
c_1|\xi-\xi_0|^{\alpha-3}\leq |\phi'''(\xi)|\leq c_2|\xi-\xi_0|^{\alpha-3},
\end{equation}
3) $\phi'''$ has a finite number of changes of monotonicity.
 \begin{lemma}\label{kpv.cu3}
 Let $\Omega$ be a bounded interval, $\phi\in \mathcal{A}_3$ and
 $$I(x,t)=\int _{\Omega} e^{i(t\phi(\xi)-x\xi)}|\phi'''(\xi)|^{1/3}d\xi.$$
 Then for any $x,t\in \rr$
 \begin{equation}\label{est.kpv.3}
|I(x,t)|\leq c_{\phi} |t|^{-1/3},
\end{equation}
where $c_\phi$ depends only on the constants involved in the definition of class $\mathcal{A}_3$.
 \end{lemma}

In the following we will write $a\lesssim b$ if there exists a positive constant $C$ such that $a\leq Cb$. Similar for $a \gtrsim b$. Also we will write $a\sim b$ if $C_1b\leq a\leq C_2 b$ for some positive constants $C_1$ and $C_2$.

\begin{proof}
We observe that since $\Omega$ is bounded we only  need to consider the case when $t$ is large.

\textbf{ Case 1}: $0< m \leq  |\phi^{'''}(\xi)| \leq M$. \\
We apply Van der Corput's Lemma with $k=3$ to the phase function $\phi(\xi)-x\xi/t$ and to $\psi=|\phi'''|^{1/3}$.
Then
$$|I(x,t)| \leq C (tm)^{-\frac{1}{3}}  (\| \psi \|_{L^\infty(\Omega)}+\| \psi '\|_{L^1(\Omega)}).$$
Since $\phi{'''}$ has a finite number of changes of monotonicity we deduce that
$\phi^{(4)}$ changes the sign finitely many times  and then
$$\| \psi '\|_{L^1(\Omega)} = \frac{1}{3} \int_{\Omega} \Big| ( \phi^{'''}(\xi))^{-\frac{2}{3}} \phi^{(4)}(\xi)\Big| d \xi  \leq \frac{1}{3}m^{-\frac{2}{3}} \int_{\Omega} |\phi^{(4)}(\xi) | d \xi\leq C(m,M).$$
Hence
$$|I(x,t)| \leq C(M,m) t^{-\frac{1}{3}}.$$

 \textbf{Case 2}: $0\leq |\phi^{'''}(\xi)|<M$. \\
 Using the assumptions on $\phi$ we can assume that  there exists only one point $\xi_0 \in \overline{\Omega}$ such that $\phi{'''}(\xi_0)=0$.
  Notice that if $\phi \in \mathcal{A}_3$, then any translation and any linear perturbation of $\phi$ (i.e. $\phi(\xi -\xi_0) +a\xi+b$) is still in  $\mathcal{A}_3$ and the conditions in the definition of set $\mathcal{A}_3$ are verified with the same constants as $\phi$. Therefore we can assume that $\xi_0 =0$ and $\phi'(\xi_0)=0$. Moreover let us assume that as $\xi\sim 0$, $|\phi'(\xi)|\sim |\xi|^{\alpha}$ and $|\phi{'''}(\xi)| \sim |\xi| ^{\beta}$ for some numbers $\alpha \geq 2$ and $\beta>0$.


We distinguish now two cases depending on the behavior of $\phi'$ near $\xi=0$.
If $\alpha\geq 4$ then $|\phi^{(k)}(\xi)|\sim |\xi|^{\alpha-k}$ as $\xi\sim 0$ for $k=2,3$ and, in particular $\beta=\alpha-3$. The case $\alpha=3$ cannot appear since then $\beta=\alpha-3$ and $\phi^{'''}$ does not vanish at $\xi=0$.
For $\alpha =2$, $|\phi'(\xi)|\sim |\xi|$, $|\phi''(\xi)|\sim 1$ as $\xi\sim 0$ and the third derivative satisfies  $|\phi'''(\xi)|\sim |\xi|^\beta$ as $\xi \sim 0$ for some positive integer $\beta$.
This last case occurs for example when $\phi'(\xi)=\xi+\xi^3$.  In all cases $\beta \geq \alpha -3$.

We split $\Omega$ as follows
$$I(x,t)= \int_{|\xi|\leq \epsilon} e^{i(t\phi(\xi )-x \xi )}|\phi {'''}(\xi)|^{\frac{1}{3}}d \xi +
\int_{|\xi|\geq \epsilon} e^{i(t\phi(\xi )-x \xi )}|\phi {'''}(\xi)|^{\frac{1}{3}}d \xi = I_1+ I_2.$$

Since $\xi=0$ is the only point where the third derivative vanishes we have that outside an interval that contains the origin $\phi'''$ does not vanish. Thus $I_2$
 can be treated as in the first case.

 Let us now estimate the first term $I_1$. We
define $\Omega_j, 1\leq j \leq 3$, as follows
$$\Omega_1 = \{ \xi \in \Omega | |\xi| \leq \min(\epsilon, |t|^{-1/\alpha}) \}, $$
$$\Omega_2 = \left\{ \xi \in \Omega - \Omega_1| |\xi| \leq \epsilon, \text{ and } \left| \phi'(\xi) - \frac{x}{t} \right|  \leq \frac{1}{2}  \left| \frac{x}{t}  \right| \right\}, $$
$$\Omega_3 = \{ \xi \in \Omega-(\Omega_1 \cup \Omega_2) | |\xi| \leq \epsilon \}.$$

In the case of $\Omega_1$ we use that for some $\beta \geq 1$, the third derivative of $\phi$ satisfies $c_1|\xi|^\beta\leq |\phi'''(\xi)|\leq c_2 |\xi|^\beta$ for $|\xi|<\epsilon$.
We get
$$\int_{\Omega_1} |\phi{'''}(\xi)| ^{\frac{1}{3}} d\xi \leq c_2 ^{\frac{1}{3}} \int_{\Omega_1} |\xi| ^{\frac{\beta}{3}} d\xi \leq C|\Omega_1|t^{-\frac{\beta}{3\alpha}}\leq C|t|^{-\frac 1\alpha-\frac \beta {3\alpha}}\leq C|t|^{-1/3},$$
where the last inequality holds since $\alpha\leq \beta+3$ and $|t|\geq 1$.

\medskip
In the case of the  integral on $\Omega_2$ we assume  that $x\neq 0 $ since otherwise $\Omega_2$ has measure zero.
Observe
 that for $\xi \in \Omega_2$ we have
$$\pm|\phi'(\xi)| \mp \left|  \frac{x}{t} \right| \leq  \left|\phi'(\xi) -  \frac{x}{t} \right| \leq  \frac{1}{2}\left|  \frac{x}{t} \right|,$$
which implies that
$$ \frac{1}{2}\Big|  \frac{x}{t} \Big|\leq |\phi'(\xi)| \leq \frac{3}{2}\Big| \frac{x}{t}\Big|.$$
Since  ${|\phi'(\xi)| \sim |\xi|^{\alpha -1}}$ we have that ${|\xi| \sim |  {x}/{t} |^{\frac{1}{\alpha - 1}}}$.
Then ${|\phi{'''}(\xi)| \sim |\xi|^\beta\sim |  {x}/{t} |^{\frac{\beta}{\alpha - 1}}}$ and
\begin{equation*}\label{minpoz}
\min_{\xi \in \Omega_2} |\phi{'''}(\xi)| > 0.
\end{equation*}
Applying Van der Corput's Lemma with $k=3$ and using that $\phi^{(4)}$ changes the sign finitely many times   we obtain that
\begin{align*}
\Big| \int_{\Omega_2} e^{i(t\phi(\xi)- x\xi)} & |\phi{'''}(\xi)|^{\frac{1}{3}} d \xi      \Big| \leq  C (\min_{\xi \in \Omega_2} |\phi{'''}(\xi)| |t|)^{- \frac{1}{3}}   \Big(  \| |\phi{'''}(\xi)|^{ \frac{1}{3}} \|_{L^\infty(\Omega_2)}  + \| (|\phi{'''}(\xi)|^{ \frac{1}{3}})' \|_{L^1(\Omega_2)}\Big)\\
&= C (\min_{\xi \in \Omega_2} |\phi{'''}(\xi)| )^{- \frac{1}{3}} |t|^{- \frac{1}{3}}  \Big(  \max_{\xi \in \Omega_2} |\phi{'''}(\xi)|^{ \frac{1}{3}}   + \frac{1}{3}\int_{\Omega_2} | \phi{'''}(\xi)|^{- \frac{2}{3}} |\phi^{(4)}(\xi)|d\xi\Big)\\
&\leq C (\min_{\xi \in \Omega_2} |\phi{'''}(\xi)| )^{- \frac{1}{3}}   \max_{\xi \in \Omega_2} |\phi{'''}(\xi)|^{ \frac{1}{3}}|t|^{- \frac{1}{3}}.
\end{align*}
Since on $\Omega_2$,  ${|\phi^{'''}(\xi)|\sim |  {x}/{t} |^{\frac{\beta}{\alpha - 1}}}$, there exists a positive constant $C$ such that
$$\max_{\xi \in \Omega_2} |\phi^{'''}(\xi)|^{ \frac{1}{3}}\leq C (\min_{\xi \in \Omega_2} |\phi^{'''}(\xi)| )^{\frac{1}{3}},$$
which gives us the desired estimates on the integral on $\Omega_2$.

\medskip
Now, we estimate the integral on $\Omega_3$. Observe that we have to consider the case $|t|^{-1/\alpha}< \epsilon$, otherwise  $\Omega_2 = \Omega_3 =\emptyset$. In particular, for
$\xi \in \Omega_3$, we have $|t|^{-1/\alpha}< \xi < \epsilon$.
Integrating by parts the integral on $\Omega_3$ satisfies
\begin{align}\label{inegalitate1}
 \Big| \int_{\Omega_3} &e^{i(t\phi(\xi)- x\xi) }  | \phi{'''}(\xi) |^{\frac{1}{3}} d \xi      \Big| = \frac{1}{|t|}
\Big|
\int_{ \Omega_3 } ( e^{ i( t \phi ( \xi )- x \xi ) } ) {'}
\frac{ |\phi{'''}(\xi)|^{ \frac{1}{3}} }
{ \phi'(\xi)-\frac{x}{t} }
d \xi
\Big| \\
\nonumber&\leq \frac{1}{|t|} \Big| \pm e^{ i(t\phi(\xi)- x \xi)} \frac{|\phi{'''}(\xi)|^{\frac{1}{3}}}{ \phi'(\xi)-\frac{x}{t}} \big |_{\partial \Omega_3} \Big| \\
\nonumber&\quad+
 \frac{1}{|t|} \Big|
  \int_{\Omega_3} e^{ i(t\phi(\xi)- x \xi)}
  \frac{ \frac{1}{3}|\phi{'''}(\xi)|^{-\frac{2}{3}} \phi^{(4)}(\xi)(\phi'(\xi) - \frac{x}{t}) - |\phi{'''}(\xi)|^{\frac{1}{3}} \phi{''}(\xi) }
  {\left( \phi{'}(\xi)- \frac{x}{t}\right)^2 }
  d\xi \Big|\\[10pt]
\nonumber&\leq \frac{2}{|t|}  \max_{\xi \in \Omega_3}
\frac{ |\phi{'''}(\xi)|^{\frac{1}{3}}}
{ \left| \phi'(\xi)-\frac{x}{t} \right| }+
\frac{1}{3|t|} \int_{\Omega_3}
\frac{ |\phi{'''}(\xi)|^{-\frac{2}{3}} | \phi{(4)}(\xi)| }
{\left| \phi{'}(\xi)- \frac{x}{t}\right| }
+ \frac{1}{|t|} \int_{\Omega_3}
\frac{|\phi{'''}(\xi)|^{\frac{1}{3}} |\phi{''}(\xi)| }
{\left( \phi{'}(\xi)- \frac{x}{t}\right)^2} d\xi .
\end{align}

In the following we obtain upper bounds for all terms in the right hand side of  \eqref{inegalitate1}.
Since on $\Omega_3$, $|\phi'(\xi)-x/t|\geq |x/2t|$, there exists a positive constant $c$ such that
$$\left| \phi{'}(\xi)- \frac{x}{t}\right| > c |\phi'(\xi)| \geq c |\xi|^{\alpha -1}, \ \forall \xi \in \Omega_3.$$

 In the case of the first term
\begin{equation}\label{pas.omega2}
\frac{1}{|t|}  \sup_{\xi \in \Omega_3}
\frac{ |\phi{'''}(\xi)|^{\frac{1}{3}}}
{ \left| \phi'(\xi)-\frac{x}{t} \right| }
\leq \frac{C}{|t|} \sup_{\xi \in \Omega_3} \frac{|\xi|^{\frac{\beta}{3}  }}{ |\xi|^{\alpha -1}}
=  \frac{C}{|t|} \sup_{\xi \in \Omega_3} |\xi|^{\frac \beta 3- \alpha+1 }\leq |t|^{-1/3},
\end{equation}
since $|\xi|\leq \eps\leq 1$ and $|\xi|^{\beta/3-\alpha+1}\leq |\xi|^{(\alpha-3)/3-\alpha+1}=|\xi|^{-2\alpha/3}\leq |t|^{2/3}$.


The second term satisfies
$$
 \frac{1}{|t|} \int_{\Omega_3}\frac{ \frac{1}{3}|\phi{'''}(\xi) |^ { - \frac{2}{3} } | \phi^{(4)}(\xi)| }
{ \left| \phi{'}(\xi)- \frac{x}{t} \right| } d\xi  \leq  \frac{C}{|t|}  \int_{\Omega_3}
\frac{ |\xi|^{-2\beta/3}  }{ |\xi|^{\alpha -1} }
| \phi^{(4)}(\xi)| d\xi\leq  \frac{C}{|t|}  \int_{\Omega_3}
|\xi|^{\frac {-2\beta}3-\alpha+1}  |\phi^{(4)}(\xi)|  d\xi .
$$
Integrating by parts, applying the triangle inequality and using the definition of $\Omega_3$ we get
\begin{align*}
\int_{\Omega_3}
|\xi|^{\frac {-2\beta}3-\alpha+1}  |\phi^{(4)}(\xi)|  d\xi&\lesssim \sup_{\Omega_3} |\xi|^{\frac {-2\beta}3-\alpha+1}  |\phi{'''}(\xi)| +\int_{\Omega_3}|\xi|^{\frac {-2\beta}3-\alpha}  |\phi{'''}(\xi)| d\xi\\
&\lesssim \sup_{\Omega_3} |\xi|^{\frac {\beta}3-\alpha+1} +\int_{\Omega_3} |\xi|^{\frac {\beta}3-\alpha}  d\xi \\
&\lesssim \sup_{\Omega_3} |\xi|^{\frac {\beta}3-\alpha+1}\leq |t|^{2/3},
\end{align*}
where the last inequality follows as in \eqref{pas.omega2}.

The last term in \eqref{inegalitate1} can be estimated as follows
\begin{align*}
\int_{\Omega_3}
\frac{|\phi{'''}(\xi)|^{\frac{1}{3}} |\phi{''}(\xi)| }{\left( \phi{'}(\xi)- \frac{x}{t}\right)^2} d\xi\lesssim
\int_{\Omega_3} \frac{|\xi|^{\beta/3+\alpha-2}}{|\xi|^{2(\alpha-2)}}=\int_{\Omega_3} |\xi|^{\beta/3-\alpha}
\lesssim \sup_{\Omega_3} |\xi|^{\frac {\beta}3-\alpha+1}\leq |t|^{2/3}.
\end{align*}
Putting together the estimates for the  terms in the right hand side of \eqref{inegalitate1} we obtain that the integral on $\Omega_3$ also decays as $|t|^{-1/3}$.

The proof is now finished.
\end{proof}

%
%

\section{Proof of the main result}\label{main.section}
\setcounter{equation}{0}

In this section we prove the main result of this paper. In order to do this, we will follow the ideas of 
  \cite{MR2049025} in the case of a discrete operator. Let us consider the system
 \begin{equation}\label{system.general}
 \left\{
\begin{array}{ll}
iU_t+AU=0,\\
U(0)=\varphi,
\end{array}
\right.
\end{equation}
where $U(t)=(u(t,j))_{j\neq 0}$ and operator $A$ is given by \eqref{matrix.a}.
We compute explicitly the resolvent $(A-\lambda I)^{-1}$, we obtain a limiting absorption principle and  finally we prove the main result of this paper Theorem \ref{main}.

\subsection{The resolvent.}
We start by localizing  the spectrum of operator $A$ and computing  the resolvent $R(\lambda)=(A-\lambda I)^{-1}$. We use some classical results on difference equations.

%
%
%
%
%
%

\begin{theorem}\label{spectrum}
For any $b_1$ and $b_2$ positive the spectrum of operator $A$ satisfies
\begin{equation}\label{spec}
\sigma(A)=[-4\max\{b_1^{-2},b_2^{-2}\},0].
\end{equation}
\end{theorem}

\begin{proof}
Since $A$ is self-adjoint we have that
$$\sigma(A)\subset [\inf_{\|u\|_{l^2(\zz^*)\leq 1}}(Au,u),\sup_{\|u\|_{l^2(\zz^*)\leq 1}}(Au,u)].$$
Explicit computations show that
$$(Au,u)=-b_1^{-2}\sum_{j\leq -1}(u_j-u_{j-1})^2-\frac 1{b_1^2+b_2^2}(u_{-1}-u_1)^2-b_2^{-2}
\sum_{j\geq 1}(u_{j+1}-u_{j})^2.$$
It is easy to see that $(Au,u)\leq 0$ and
$$(Au,u)\geq -2\max \{b_1^{-2},b_2^{-2}\}\sum _{j\in \zz^*}(u_j^2+u_{j+1}^2)=-4\max \{b_1^{-2},b_2^{-2}\}\sum _{j\in \zz^*}u_j^2.$$

In order to prove that the spectrum is continuous we need to prove that for any $\lambda\in [-4\max\{b_1^{-2},b_2^{-2}\},0]$ we can find $u_n\in l^2(\zz^*)$ with $\|u_n\|_{l^2(\zz^*)}\leq 1$ such that $\|(A-\lambda I) u_n\|_{l^2(\zz^*)}$ tends to zero. To fix the ideas let us assume that $b_2\leq b_1$ and $\lambda \in [-4b_2^{-2},0]$. We construct $u_n$ such that all its components $u_{n,j}$, $j\leq -1$, vanish. Thus for such $u_n$'s we have that
$$(Au_n)_j=b_2^{-2}(\Delta_d u_n)_j,\ j\geq 1.
$$
Using the fact that any $\lambda \in [-4b_2^{-2},0]$ belongs to $\sigma(b_2^{-1}\Delta_d)$ we can construct sequences $(u_{n,j})_{j\geq 1}$ such that $\|u_n\|_{l^2(\zz^*)}\leq 1$ and $\|(A-\lambda I) u_n\|_{l^2(\zz^*)}\rightarrow 0$. This implies that $\lambda\in \sigma(A)$ and the proof is finished.
\end{proof}

Before computing the resolvent $(A-\lambda I)^{-1}$ we need some results for  difference equations.
\begin{lemma}\label{eq.pe.z.plus}
For any $\lambda \in \cn\setminus [-4,0]$ and $g\in l^2(\zz^*)$, any solution $f\in l^2(\zz^*)$ of
$$\Delta_d f(j)-\lambda f(j)=g(j), \quad j\neq 0$$
with $f(0)$  prescribed is given by
\begin{equation}\label{formula.1}
f(j)=\alpha r^{|j|}+\frac 1{2r-2-\lambda}\sum _{k\in \zz^*}r^{|j-k|}g(k)
\end{equation}
where $\alpha$ is determined by $f(0)$ and $r$ is the unique solution with $|r|<1$ of
$$r^2-2r+1=\lambda r.$$
Moreover
$$f(j)=f(0)r^{|j|}+\frac 1{r-r^{-1}}\sum _{k}(r^{|j-k|}-r^{|j|+|k|})g(k), \quad j \neq 0.$$
\end{lemma}

\begin{proof}
Let us consider the case when $j \geq 1$, the other case $j \leq -1$ can be treated similarly.
Writing the equation satisfied by $f$ we obtain that
$$f(j+1)-(2+\lambda)f(j)+f(j-1)=g(j),\quad j\geq 1.$$
This is an inhomogeneous difference equation whose solutions are written as the sum between a particular solution and the general solution for the homogeneous difference equation
$$f(j+1)-(2+\lambda)f(j)+f(j-1)=0,\quad j\geq 1.$$
Let us denote by $r_1$ and $r_2$, $|r_1|\leq |r_2|$, the two solutions of the second order equation
$$r^2-(2+\lambda)r+1=0.$$
Since $2+\lambda\in \cn\setminus [-2,2]$ we have that $r_1$ and $r_2$ belong to $\cn\setminus \rr$ and
more than that
$|r_1|<1<|r_2|$. Thus
we obtain that
\begin{equation}\label{formulaf(j)}
f(j)=\alpha r_1^j+\beta r_2^j+\frac 1{2r-2-\lambda}\sum _{k\in\zz^*}r_1^{|j-k|}g(k).
\end{equation}
Since $f$ is an $l^2(\zz^+)$ function we should have $\beta=0$. Then formula \eqref{formula.1} holds. The last identity is obtained by putting $j=0$ in \eqref{formulaf(j)} and using that $2r-2-\lambda=r-r^{-1}$.
\end{proof}

As an application of the previous Lemma we have the following result.
\begin{lemma}\label{eq.on.whole.z}
Set $Z_1=\zz\cap (-\infty,-1]$ and $Z_2=\zz\cap[1,\infty)$. For any
$\lambda \in \cn\setminus [-4\max\{b_1^{-2},b_2^{-2}\} ,0]$ and  $g\in l^2(\zz^*)$, any solution $f\in l^2(\zz)$ of
$$b_s^{-2}\Delta_d f(j)-\lambda f(j)=g(j), \quad j\in Z_s,$$
with $f(0)$ prescribed  is given by
\begin{equation}\label{formula.2}
f(j)=\alpha_sr_s^{|j|}+\frac {b_s^2}{2r_s-2-\lambda b_s^2}\sum _{k\in Z_s}r_s^{|j-k|}g(k),\quad j\in Z_s, s\in \{1,2\}
\end{equation}
where for $s\in \{1,2\}$, constant $\alpha_s$ is determined by $f(0)$ and $r_s$ is the unique solution with $|r_s|<1$ of
$$r_s^2-2r_s+1=\lambda r_s b_s^{2}.$$
Moreover
\begin{equation}\label{formula.cuzero}
f(j)=f(0)r_s^{|j|}+\frac {b_s^{2}}{r_s-r_s^{-1}}\sum _{k\in Z_s}(r_s^{|j-k|}-r_s^{|j|+|k|})g(k), j\in Z_s.
\end{equation}
\end{lemma}

The  proof of this lemma consists in just applying Lemma \ref{eq.pe.z.plus} to the difference equations in $Z_1$ and $Z_2$.

\begin{lemma}\label{resolvent}
Let $\lambda \in \cn\setminus [-4\max\{b_1^{-2},b_2^{-2}\} ,0]$. For any $g\in l^2(\zz^*)$ there exists a unique solution $f\in l^2(\zz^*)$ of the equation
$(A-\lambda I)f=g$. Moreover, it is given by the following formula
\begin{align}\label{rez}
f(j)=\frac {-r_s^{|j|}}{b_2^{-2}(1-r_2)+b_1^{-2}(1-r_1)}&\Big[\sum _{k\in Z_1}r_1^{|k|}g(k)+\sum _{k\in Z_2}r_2^{|k|}g(k)\Big]\\
\nonumber&+\frac{b_s^2}{r_s-r_s^{-1}}\sum _{k\in Z_s}(r_s^{|j-k|}-r_s^{|j|+|k|})g(k),\quad j\in Z_s,
\end{align}
 where for $s\in \{1,2\}$, $r_s=r_s(\lambda)$ is the unique solution with $|r_s|<1$ of the equation
 $$r_s^2-2r_s+1=\lambda b_s^2 r_s.$$
\end{lemma}

\begin{proof}
Any solution of $(A-\lambda I)f=g$ satisfies
$$
\left\{
\begin{array}{l}
\Delta_d f(j)-b_s^{2}\lambda f(j)=b_s^{2}g(j), \quad j\in Z_s,\\[10pt]
b_1^{-2}(f(-1)-f(0))=b_2^{-2}(f(0)-f(1)),
\end{array}
\right.
$$
where $f(0)$ is artificially introduced  in order to write the system in a convenient form that permits us to apply Lemma \ref{eq.on.whole.z}.

Using \eqref{formula.cuzero} we obtain
$$f(-1)=f(0)r_1-{b_1^{2}}\sum_{k\in Z_2}r_1^{|k|}g(k)$$
and
$$f(1)=f(0)r_2-{b_2^{2}}\sum_{k\in Z_2}r_2^{|k|}g(k).$$
The coupling condition gives us that
$$f(0)=\frac{-1}{b_1^{-2}(1-r_1)+b_2^{-2}(1-r_2)}\sum _{s=1,2,\ k\in Z_s} r_s^{|k|}g(k).$$
Introducing this formula in \eqref{formula.cuzero} we obtain the explicit formula of the resolvent.
\end{proof}

\subsection{Limiting absorption principle}In this subsection we write a limiting absorption principle.
From Lemma \ref{resolvent} we know that for any
 $\lambda \in \cn\setminus [-4\max\{b_1^{-2},b_2^{-2}\} ,0]$ and  $\varphi\in l^2(\zz^*)$
 there exists $R(\lambda) \varphi =(A-\lambda)^{-1} \varphi \in l^2(\zz^*)$ and it is given by
\begin{align}\label{rez.n}
(R(\lambda) \varphi)(j)=\frac {-r_s^{|j|}}{b_2^{-2}(1-r_2)+b_1^{-2}(1-r_1)}&\Big[\sum _{k\in I_1}r_1^{|k|} \varphi(k)+\sum _{k\in I_2}r_2^{|k|} \varphi(k)\Big]\\
\nonumber&+\frac{b_s^2}{r_s-r_s^{-1}}\sum _{k\in I_s}(r_s^{|j-k|}-r_s^{|j|+|k|}) \varphi(k),\quad j\in Z_s,
\end{align}
 where $r_s=r_s(\lambda)$, $s\in \{1,2\}$, is the unique solution with $|r_s|<1$ of the equation
 $$r_s^2-2r_s+1=\lambda b_s^2 r_s.$$

Let us now consider $I=[-4\max\{b_1^{-2},b_2^{-2}\},0]$. As we proved in Theorem \ref{spectrum} we have that
$\sigma(A)= I$.  For any $\omega \in I$ and $\eps\geq 0$ let us denote by $r_{s,\eps}^{\pm}$ the unique solution  with modulus less than one of
$$r^2-2r+1=(\omega\pm i\eps)b_s^2 r.$$
Denoting $r^+_{s,\eps}=\exp(z^+_{s,\eps})$ with $z^+_{s,\eps}=a^+_{s,\eps}+i\tilde a^+_{s,\eps}$, $a^+_{s,\eps}<0$ and $\tilde a^+_{s,\eps}\in [-\pi,\pi]$ we obtain by taking the imaginary part in the equation satisfied by $r^+_{s,\eps}$ that
$$(\exp(a^+_{s,\eps})-\exp(-a^+_{s,\eps}))\sin (\tilde a^+_{s,\eps})=\eps b_s^2.$$
Thus $\tilde a^+_{s,\eps}\in [-\pi,0]$.
A similar result holds for $r^- _{s,\epsilon}$, $\tilde a^-_{s, \epsilon} \in [0,\pi]$.

Let us set $r_s^{\pm}=\lim  _{\epsilon\downarrow 0}r^{\pm}_{s,\eps}.$ Using the sign of the imaginary part of $r^{\pm}_{s,\eps}$ we obtain that  $r_s^{\pm}$ are the solutions with $\Im(r_s^+)\leq 0\leq \Im(r_s^{-})$  of the equation
$$r^2-2r+1=\omega b_s^2 r.$$
Also, using that $r^{-}_{s,\eps}=\overline{r_{s,\eps}^+}$ we obtain $r^{-}_s=\overline{r^+_s}$.

\medskip
For any $\omega\in J=I\setminus\{-4b_1^{-2},-4b_2^{-2},0\}$ and $\varphi\in l^1(\zz^*)$ let us set
\begin{align*}
(R^{\pm}(\omega) \varphi)(j)=&\frac {-(r_s^{\pm})^{|j|}}{b_2^{-2}(1-r_2^{\pm})+b_1^{-2}(1-r_1^{\pm})}\Big[\sum _{k\in I_1}(r_1^{\pm})^{|k|} \varphi(k)+\sum _{k\in I_2}(r_2^{\pm})^{|k|} \varphi(k)\Big]\\
\nonumber&+\frac{b_s^2}{r_s^{\pm}-(r_s^{\pm})^{-1}}\sum _{k\in I_s}((r_s^{\pm})^{|j-k|}-(r_s^{\pm})^{|j|+|k|}) \varphi(k),\quad j\in Z_s.
\end{align*}
We will prove that $R^{\pm}(\omega)$ are well defined as bounded operators from $l^1(\zz^*)$ to $l^\infty(\zz^*)$.
We point out that we cannot define $R^\pm(\omega)$ for $\omega \in \{ -4 b_1 ^{-2}, -4b_2^{-2},0 \} $ since for $\omega=0$ we have $r_1=r_2=1$ and for $\omega=4 b_s ^{-2}, s \in \{1,2\},$ we have $r_s=-1$. We also emphasize that $R^{-}(\omega)\varphi=\overline{R^{+}(\omega)\overline\varphi}$. This is a consequence of the fact that for any $\omega\in I$, $r^{-}_s(\omega)=\overline{r^+_s}(\omega)$.
Formally, the above operator equals  $R(\omega\pm i\eps)$ with $\eps=0$. We point out that as operators on $l^2(\zz^*)$, $R(\omega\pm i\eps)$ are defined for any $\omega\in I$ but only if $\eps\neq 0$.

\begin{lemma}\label{lap}
For any $\varphi\in l^1(\zz^*)$ operator $\exp(itA)$ satisfies
\begin{equation}\label{formula.lap}
e^{itA}\varphi =\frac 1{2i\pi}\int _I e^{it\omega}[R^{+}(\omega)-R^{-}(\omega)]\varphi\, d\omega.
\end{equation}
\end{lemma}

\begin{proof}To clarify the ideas behind the proof we divide it in   several steps.

Step 1. Let $I_1$ be a bounded interval such that $I\subset I_1$.
 There exists a constant
\begin{equation}\label{C(omega)}
C(\omega )=\frac {1}{|\omega|^{1/2}}+\frac 1{|\omega b_1^2+4|^{1/2}}+\frac 1{|\omega b_2^2+4|^{1/2}}\in L^1(I_1)
\end{equation}
such that
for all $\omega \in I_1\setminus\{-4b_1^{-2},-4b_2^{-2},0\}$
 the following inequality
$$|(R(\omega\pm i\eps)\varphi)(n)|\lesssim C(\omega)\| \varphi\|_{l^1(\zz^*)},\text{ for all } \varphi\in l^1(\zz^*)\ \text{and } n \in \zz^*,$$
holds uniformly on small enough $\eps$.

Step 2.  For any $\omega\in J$,  $R^{\pm}(\omega)$ are bounded operators from $l^1(\zz^*)$ to $l^\infty(\zz^*)$
and
$$\|R^{\pm}(\omega)\|_{l^1(\zz^*)-l^\infty(\zz^*)}\lesssim C(\omega).$$

Step 3. For any $\omega\in J$, $\varphi\in l^1(\zz^*)$  and $n\in \zz^*$  the following holds
$$\lim_{\eps \downarrow 0}(R(\omega\pm i \eps) \varphi)(n) =(R^{\pm}(\omega) \varphi)(n).$$

Step 4. For any $\varphi\in l^1(\zz^*)$ and $n\in \zz^*$ we have
$$\lim_{\eps \downarrow 0}\int_{I}e^{it\omega} (R(\omega\pm i\eps)\varphi)(n)d\omega =
\int_{I}e^{it\omega} (R^{\pm}(\omega)\varphi)(n)d\omega .$$

Step 5. For any $\varphi\in l^1(\zz^*)$
$$e^{itA}\varphi=\frac 1{2i\pi}\int _I e^{it\omega}[R^{+}(\omega)-R^{-}(\omega)]\varphi d\omega. $$

Proof of Step 1. Observe that for any $\omega\in \rr$ and $\eps>0$ we have
\begin{align*}
|(R(\omega\pm & i\eps)\varphi)(n)|\\
&\lesssim \|\varphi\|_{l^1(\zz^*)}\Big(
\frac {1}{|b_2^{-2}(1-r_{2,\eps}^\pm)+b_1^{-2}(1-r_{1,\eps}^\pm)|}+
\frac{1}{|r_{1,\eps}^{\pm}-(r_{1,\eps}^{\pm})^{-1}|}+\frac{1}{|r_{2,\eps}^{\pm}-(r_{2,\eps}^{\pm})^{-1}|}\Big).
\end{align*}
Solution $r_{s,\eps}^{\pm}$ of equation $r^2-2r+1=(\omega\pm i\eps)b_s^2 r$ satisfies
$$\frac 1{|r_{s,\eps}^\pm|}-|r_{s,\eps}^\pm|\leq \left| r_{s,\eps}^\pm-\frac 1{r_{s,\eps}^\pm} \right| = b_s|\omega\pm i\eps|^{1/2}.$$
Then for all  $\omega\in I_1$ and $ \eps$ small enough we have
$$|r_{s,\eps}^\pm|\geq \frac {2}{b_s|\omega\pm i\eps|^{1/2}+(b_s^2|\omega\pm i\eps|+4)^{1/2}}\geq C>0$$
and
$$|r_{s,\eps}^\pm|\leq \frac 1{|r_{s,\eps}^\pm|}+\Big|r_{s,\eps}^\pm-\frac 1{r_{s,\eps}^\pm}\Big|\leq C_1<\infty.
$$
Thus for any $\omega\in I_1$ we have
$$\frac{1}{|r_{s,\eps}^{\pm}-(r_{s,\eps}^{\pm})^{-1}|}\lesssim \frac{1}{|1-r_{s,\eps}^{\pm}||1+r_{s,\eps}^{\pm}|}\lesssim
\frac{1}{|1-r_{s,\eps}^{\pm}|}+\frac{1}{|1+r_{s,\eps}^{\pm}|}.$$
Using the equation satisfied by $r_{s,\eps}^{\pm}$ we find that
$$|1-r_{s,\eps}^{\pm}|=b_s|\omega\pm i\eps|^{1/2}|r_{s,\eps}^{\pm}|\gtrsim |\omega\pm i\eps|^{1/2}\geq
|\omega|^{1/2}$$
and
$$|1+r_{s,\eps}^{\pm}|=|(\omega\pm i\eps)b_s^2+4|^{1/2}|r_{s,\eps}^{\pm}|\gtrsim |(\omega\pm i\eps)b_s^2+4|^{1/2}\geq |\omega b_s^2+4|^{1/2}.$$
Putting together the above estimates for the roots $r_{s,\eps}^{\pm}$ we find that for all $\omega\in I_1$ and $\eps$ small enough the following holds
$$\frac{1}{|r_{1,\eps}^{\pm}-(r_{1,\eps}^{\pm})^{-1}|}+\frac{1}{|r_{2,\eps}^{\pm}-(r_{2,\eps}^{\pm})^{-1}|}\lesssim
\frac {1}{|\omega|^{1/2}}+\frac 1{|\omega b_1^2+4|^{1/2}}+\frac 1{|\omega b_2^2+4|^{1/2}}.$$

We now prove that $$\frac{1}{ |b_2^{-2}(1-r_{2,\eps}^\pm)+b_1^{-2}(1-r_{1,\eps}^\pm) |} \lesssim \frac{1}{|\omega|^{1/2}}.$$
We recall that the sign of the imaginary parts of  $r_{1,\eps}^{\pm}$ and
$r_{2,\eps}^{\pm}$ is the same. Also, since $|r_{s,\eps}^{\pm}|<1$, the real parts of $1-r_{1,\eps}^{\pm}$ and $1-r_{2,\eps}^{\pm}$ are positive. These properties of the roots imply that
$$|b_2^{-2}(1-r_{2,\eps}^\pm)+b_1^{-2}(1-r_{1,\eps}^\pm)|\geq b_2^{-2}|1-r_{2,\eps}^\pm|+b_1^{-2}|1-r_{1,\eps}^\pm|\gtrsim |\omega|^{1/2}.$$

Putting together the above results we obtain that Step 1 is satisfied with $C(\omega)$ given by \eqref{C(omega)} uniformly on all $\eps>0$ sufficiently small.

\medskip
Step 2 follows as Step 1 by putting $\eps=0$ and replacing $r_{s,\eps}^\pm$ with $r_{s}^\pm$.

\medskip
{Proof of Step 3}.
We write $$R(\omega\pm i\eps)\varphi(n)=\sum_{k\in \zz^*}R(\omega\pm i\eps,n,k)\varphi(k),$$
where $R(\omega\pm i\eps,n,k)$ collects all the coefficients in front of $\varphi(k)$ in  formula \eqref{rez}.

Using that, for any $\omega \in J$, $r_{s,\eps}^{\pm}(\omega)\rightarrow r_s^{\pm}(\omega)$ we obtain that
$R(\omega\pm i\eps,n,k)\varphi(k)\rightarrow R^\pm(\omega,n,k)\varphi(k).$
Since for any $\omega\in J$ and $\eps$ small enough we have the uniform bound
$$|R(\omega\pm i\eps,n,k)\varphi(k)|\leq C(\omega)|\varphi(k)|, \forall k \in \zz^*,$$
we can apply Lebesgue's dominated convergence theorem to conclude that
$$\sum_{k\in \zz^*}R(\omega\pm i\eps,n,k)\varphi(k)\rightarrow \sum_{k\in \zz^*}R^\pm(\omega,n,k)\varphi(k),$$
which proves Step 3.

Step 4 follows by Lebesgue's dominated convergence theorem since we have the pointwise convergence in Step 3 and the uniform bound in Step 1.

Proof of Step 5.
Applying  Cauchy's formula we obtain that
$$e^{itA}=\frac 1{2i\pi}\int_{\Gamma}e^{it\omega}R(\omega)d\omega$$
for any curve $\Gamma$ that rounds the spectrum of operator $A$. For small parameter $\eps$ we choose in the above formula  path $\Gamma_\eps$ to be the following rectangle
\begin{align*}
\Gamma_\eps =&\{\omega\pm i\eps,\omega\in [-4\max\{b_1^{-2},b_2^{-2}\}-\eps,\eps]\}\\
&\cup\{
-4\max\{b_1^{-2},b_2^{-2}\}-\eps+i\eta,\eta\in [-\eps,\eps]\}\cup\{\eps+i\eta,\eta\in [-\eps,\eps]\}.
\end{align*}
Using the estimates for $R(\lambda)$, $\lambda\in \Gamma_\eps$ obtained in Step 1 and the convergence in  Step 4 we obtain
 that for any $\varphi\in l^1(\zz^*)$ the following holds:
$$e^{itA}\varphi=\frac 1{2\pi i}\int _{I} e^{it\omega}(R^+(\omega)-R^{-}(\omega))\varphi d\omega.$$
The proof is now complete.
\end{proof}


\subsection{Proof of the main result}
We now prove the main result of this paper.

\begin{proof}[Proof of Theorem \ref{main}.]
For any $\varphi \in l^1(\zz^*)$ Lemma \ref{lap} gives us that
$$(e^{itA}\varphi)(n)=\frac 1{2\pi i}\int _{I} e^{it\omega}(R^+(\omega)-R^{-}(\omega))\varphi (n)ds,\, n\in \zz^*,$$
where $I=[-4\max\{b_1^{-2},b_2^{-2}\},0]$. Using the fact that $R^{-}(\omega)\varphi=\overline{R^{+}(\omega)\overline\varphi}$ we obtain
$$(e^{itA}\varphi)(n)=\frac 1{\pi}\int _I e^{it\omega} ((\Im R^{+})(\omega)\varphi) (n)d\omega, \ n \in \zz^*,$$
where
 $\Im R^+$ is given by
 \begin{align*}
(\Im R^{+})(\omega)\varphi(j)=&\frac{ (R^+(\omega)\varphi)(j)-(R^{-}(\omega)\varphi)(j)}{2i}\\
\nonumber=&\sum _{k\in Z_1} \varphi(k)\Im \frac {-(r_s^+)^{|j|}(r_1^+)^{|k|}}{b_2^{-2}(1-r^+_2)+b_1^{-2}(1-r^+_1)}\\
\nonumber&+\sum _{k\in Z_2} \varphi(k)\Im \frac {-(r_s^+)^{|j|}(r_2^+)^{|k|}}{b_2^{-2}(1-r^+_2)+b_1^{-2}(1-r^+_1)}\\
\nonumber&+\sum _{k\in Z_s}  \varphi(k) \Im \frac{b_s^2}{r_s^+-(r_s^+)^{-1}}((r_s^+)^{|j-k|}-(r_s^+)^{|j|+|k|}),\quad j\in Z_s
\end{align*}
 and for  $s\in \{1,2\}$, $r_s^+$ is the root  of
$r^2-2r+1=\omega b_s^2 r$ with the imaginary part nonpositive.

\medskip
In order to prove \eqref{decay.main} it is sufficient to show the existence of a constant $C=C(b_1,b_2)$ such that
\begin{equation}\label{es.1}
\sum _{k\in Z_1}|\varphi(k)|  \Big|\int _I e^{it\omega}\Im\frac {(r_s^+)^{|j|} (r_1^+)^{|k|}} {b_2^{-2}(1-r^+_2)+b_1^{-2}(1-r^+_1)}d \omega \Big|  \leq C  (|t|+1)^{-1/3}\|\varphi\|_{l^1(\zz^*)}, \ \forall j \in \zz^*,
\end{equation}
and
\begin{equation}\label{es.2}
\sum _{k\in Z_s} |\varphi(k)| \Big|\int _I e^{it\omega}\Im \frac{(r_s^+)^{|j-k|}}{r_s^+-(r_s^+)^{-1}}d\omega\Big|  \leq C  (|t|+1)^{-1/3}\|\varphi\|_{l^1(\zz^*)}, \ \forall j \in \zz^*.
\end{equation}
The estimates for the other two terms occurring in the representation of $\Im R^+(\omega)$ are similar.

\medskip
\textbf{Step I. Proof of \eqref{es.2}.}
We prove that
\begin{equation}\label{est.4}
\sup_{j\in \zz}\Big|\int _I e^{it\omega} \Im\frac{(r_s^+)^{|j|}}{r_s^+-(r_s^+)^{-1}} d\omega \Big|\leq C(b_1,b_2) (|t|+1)^{-1/3},\quad \forall\ t\in \rr.
\end{equation}
We split $I$ as  $I=I_1\cup I_2$ where
$ I_1=[-4\max\{b_1^{-2},b_2^{-2}\},4b_s^{-2}]$ and  $ I_2=[4b_s^{-2},0]$.
If $\omega\in I_1$, the following equation $$r+\frac 1{r}=2+\omega b_s^2$$ has real roots and then
$$\int _I e^{it\omega} \Im\frac{(r_s^+)^{|j|}}{r_s^+-(r_s^+)^{-1}} d \omega =0.$$
When $\omega\in I_2$, root $r_s$ of equation $r_s+\frac 1{r_s}=2+\omega b_s^2$ has the form $r_s=e^{- i\theta}, \theta \in [0,\pi]$.

Using the change of variables $\omega =b_s^{-2}(2\cos \theta-2)$ we get
\begin{align*}
\int _{I_2} e^{it\omega} \Im\frac{(r_s^+)^{|j|}}{r_s^+-(r_s^+)^{-1}}d\omega&=
2b_s^{-2}\int _{0}^\pi e^{itb_s^{-2}(2\cos \theta-2)} \Im\frac{ e^{-i  |j|\theta}}{e^{-i\theta}-e^{i \theta}}\sin \theta d\theta\\
=&-2b_s^{-2}\int _{0}^\pi e^{itb_s^{-2}(2\cos \theta-2)}  \Im\frac{e^{-i |j|\theta}}{2i\sin \theta}\sin \theta d\theta\\
=&b_s^{-2}\int _0^{\pi} e^{itb_s^{-2}(2\cos \theta-2)} \Re e^{-i|j|\theta}d\theta\\
=&\frac { b_s^{-2}}2\int _0^{\pi} e^{itb_s^{-2}(2\cos \theta-2)} ( e^{i|j|\theta}+e^{-i|j|\theta})d\theta.
\end{align*}
Van der Corput's Lemma  applied to the phase function $\phi(\theta)=(2\cos \theta-2)b_s^{-2}+j {\theta}/{t}$ shows that
\begin{equation}\label{van.1}
\Big|\int _0^{\pi} e^{it(2\cos \theta-2)b_s^{-2}} e^{ij\theta }d\theta\Big|\leq C(b_s) (|t|+1)^{-3}, \ \forall\  t \in \rr, \forall j \in \zz
\end{equation}
The proof of \eqref{es.2} is now finished.
\medskip

\textbf{Step II. Proof of \eqref{es.1}.}
It is sufficient  to prove that
$$
\sup_{j,k\in \nn}\Big|\int _I e^{it\omega}\frac {(r_1^+)^{j}(r_2^+)^{k} }{b_2^{-2}(1-r^+_2)+b_1^{-2}(1-r^+_1)} d\omega
\Big|\leq  C(b_1,b_2)(|t|+1)^{-1/3}, \ \forall t\in \rr.
$$
To fix the ideas let us assume that $b_2\leq b_1$. We split interval $I$ as follows $I=I_1\cup I_2$ where
$I_1=[-4b_2^{-2},-4b_1^{-2}]$ and $I_2=[-4b_1^{-2},0]$. We remark that on $I_1$, $r_1^+\in \rr$ and $r_2^+\in \cn\setminus \rr$. On $I_2$ both $r_1^+$ and $r_2^+$ belong to $\cn\setminus \rr$. We prove that
\begin{equation}\label{est.i1}
\sup_{j,k\in \nn}\Big|\int _{I_1} e^{it\omega}\frac {(r_1^+)^{j}(r_2^+)^{k} }{b_2^{-2}(1-r^+_2)+b_1^{-2}(1-r^+_1)} d\omega
\Big|\leq  C(b_1,b_2)(|t|+1)^{-1/3}
\end{equation}
and
\begin{equation}\label{est.i2}
\sup_{j,k\in \nn}\Big|\int _{I_2} e^{it\omega}\frac {(r_1^+)^{j}(r_2^+)^{k} }{b_2^{-2}(1-r^+_2)+b_1^{-2}(1-r^+_1)} d\omega
\Big|\leq  C(b_1,b_2)(|t|+1)^{-1/3}.
\end{equation}

Let us set $h(\omega)=b_2^{-2}(1-r^+_2(\omega))+b_1^{-2}(1-r^+_1(\omega))$
Using the same arguments as in the proof of Lemma \ref{lap} we get that $|h(\omega)|\geq C(b_1,b_2) |\omega|^{1/2}.$ Then, on $I_1$, $|h(\omega)|\geq c >0.$
Moreover  $|h'(\omega)|\leq c_2<\infty$. Using integration by parts we obtain that
\begin{align*}
\Big|\int _{I_1} e^{it\omega}&\frac {(r_1^+)^{j}(r_2^+)^{k} }{b_2^{-2}(1-r^+_2)+b_1^{-2}(1-r^+_1)} d\omega\Big|\\
&\leq
\sup_{x\in I_1} \Big| \int_{-4b_2^{-2}}^x  e^{it\omega}(r_1^+)^{j}(r_2^+)^{k} d\omega\Big|\Big(\|1/h\|_{L^\infty(I_1)}+
\|(1/h)'\|_{L^1(I_1)}\Big)\\
&\leq  C(b_1,b_2)\sup_{x\in I_1} \Big| \int_{-4b_2^{-2}}^x  e^{it\omega}(r_1^+)^{j}(r_2^+)^{k} d\omega\Big|.
\end{align*}
A similar argument shows that
\begin{align*}
\Big| \int_{-4b_2^{-2}}^x  e^{it\omega}(r_1^+)^{j}(r_2^+)^{k} d\omega\Big|\leq
\sup_{y\leq x}\Big| \int_{-4b_2^{-2}}^y  e^{it\omega}(r_2^+)^{k} d\omega\Big|\Big(\|(r_1^+)^j\|_{L^\infty(I_1)}+ \|((r_1^+)^j)'\|_{L^\infty(I_1)}\Big).
\end{align*}
Observe that for $\omega\in I_1$, $r_1^+(\omega)$  given by
$$r_1^+(\omega)=\frac {2+b_1^2\omega-\sqrt{(2+b_1^2\omega)^2-4}}2$$
is a decreasing function.
Thus
$$\|((r_1^+)^j)'\|_{L^1(I_1)}\leq  \|(r_1^+)^j\|_{L^\infty(I_1)}\leq 1,\quad \forall j\in
\nn.$$
The proof of \eqref{est.i1} is now reduced to the following estimate:
$$\sup_{y\in I_1}\Big| \int_{-4b_2^{-2}}^y  e^{it\omega}(r_2^+(\omega))^{k} d\omega\Big|\leq C(b_1,b_2)(|t|+1)^{-1/3}, \forall k \in \nn, t \in \rr.$$
Making the change of variables $\omega=b_2^{-2}(2\cos \theta -2)$ and applying Van der Corput's Lemma as in the final step of Step I we obtain that
\begin{align*}
\Big| \int_{-4b_2^{-2}}^y  e^{it\omega}(r_2^+(\omega))^{k} d\omega\Big|=
2b_2^{-2}\Big| \int_{2\arcsin(b_2^2/y)}^{\pi}  e^{it b_2^{2}(2\cos \theta-2)}e^{-ik\theta} \sin \theta d\omega\Big|\leq
C(b_2)(|t|+1)^{-1/3}.
\end{align*}

\medskip
We now prove \eqref{est.i2}. We first make the change of variables $\omega=b_1^{-2}(2\cos \theta -2)$. Thus
\begin{align*}
\int _{I_2} e^{it\omega}\frac {(r_1^+)^{j}(r_2^+)^{k} }{b_2^{-2}(1-r^+_2)+b_1^{-2}(1-r^+_1)} d\omega=
2b_1^{-2}\int _0^\pi e^{itb_1^{-2}(2\cos \theta -2)}e^{-ij\theta}e^{-2ik\arcsin (b_2b_1^{-1}\sin \frac \theta 2)}\frac{\sin \theta} {h(\theta)}d\theta,
\end{align*}
where $h(\theta)=b_2^{-2}(1-r^+_2(\theta))+b_1^{-2}(1-r^+_1(\theta))$, $r_1^+(\theta)=e^{-i\theta}$ and $r_2^+(\theta)=e^{-2i\arcsin (b_2b_1^{-1}\sin \frac \theta 2)}$.

Using that far from $\theta=0$ function $h$ satisfies $|h(\theta)|>0$ we choose a small parameter $\eps$ and split our integral as follows:
\begin{align*}
\int _0^\pi e^{itb_1^{-2}(2\cos \theta -2)}& e^{-ij\theta}e^{-2ik\arcsin (b_2b_1^{-1}\sin \frac \theta 2)}\frac{\sin \theta} {h(\theta)}d\theta=T_1+T_2\\
&=\int _0^\eps e^{itb_1^{-2}(2\cos \theta -2)}e^{-ij\theta}e^{-2ik\arcsin (b_2b_1^{-1}\sin \frac \theta 2)}\frac{\sin \theta} {h(\theta)}d\theta\\
&\quad + \int _\eps^\pi e^{itb_1^{-2}(2\cos \theta -2)}e^{-ij\theta}e^{-2ik\arcsin (b_2b_1^{-1}\sin \frac \theta 2)}\frac{\sin \theta} {h(\theta)}d\theta.
\end{align*}

Observe that on interval $[0,\eps]$
$$\Big\|\frac{\sin \theta} {h(\theta)} \Big\|_{L^\infty(0,\eps)}+ \Big\|\frac{d}{d\theta}(\frac{\sin \theta} {h(\theta)})\Big\|_{L^1(0,\eps)}\leq M<\infty$$
and on interval $[\eps,\pi]$
$$\Big\|\frac{1} {h(\theta)}\Big\|_{L^\infty(\eps,\pi)}+\Big\|\frac{d}{d\theta}(\frac{1} {h(\theta)})\Big\|_{L^1(\eps,\pi)}\leq M<\infty.$$

Then we have the following estimates for $T_1$ and $T_2$
\begin{align*}
|T_1|
\leq M\sup _{x\in [0,\eps]}\Big| \int _0^x e^{itb_1^{-2}(2\cos \theta -2)}e^{-ij\theta}e^{-2ik\arcsin (b_2b_1^{-1}\sin \frac \theta 2)}d\theta \Big|
\end{align*}
and
\begin{align*}
|T_2|
\leq M\sup _{x\in [\eps,\pi]}\Big| \int _x^\pi e^{itb_1^{-2}(2\cos \theta -2)}e^{-ij\theta}e^{-2ik\arcsin (b_2b_1^{-1}\sin \frac \theta 2)}\sin \theta d\theta \Big|.
\end{align*}

We now apply the following lemma that we prove later.
\begin{lemma}\label{lemma.forte}
Let $a\in (0,1]$ and $0\leq \delta\leq \pi$. There exists $C(a,\delta)$ such that  for all real numbers $y$, $z$ and $t$
\begin{equation}\label{int.osc}
\Big| \int _\delta^\pi e^{it (2\cos \theta +2z\arcsin (a\sin \frac \theta 2))}e^{iy\theta}\sin \theta d\theta \Big|\leq C(a,\delta)(|t|+1)^{-1/3}
\end{equation}
and if $\delta>0$
\begin{equation}\label{int.osc.2}
\Big| \int _0^{\pi-\delta} e^{it (2\cos \theta +2z\arcsin (a\sin \frac \theta 2))}e^{iy\theta} d\theta \Big|\leq C(a,\delta)(|t|+1)^{-1/3}.
\end{equation}
\end{lemma}

We obtain that $$|T_1|\leq MC(a,\eps)(|t|+1)^{-1/3}$$
and $$|T_2|\leq MC(a,\eps)(|t|+1)^{-1/3}.$$
The proof of Theorem \ref{main} is now finished.
\end{proof}

\begin{proof}[Proof of Lemma \ref{lemma.forte}]
Since the integrals in \eqref{int.osc} and \eqref{int.osc.2} are on bounded intervals it is sufficient to prove that, for $t$ large enough, each of the integrals is bounded by $|t|^{-1/3}.$ In the case of \eqref{int.osc} we will consider the case $\delta=0$ since the proof for $\delta>0$ is similar.

Let us denote by $\psi$ either the function $\chi_{(0,\pi-\delta)}$ or $\sin \theta$.
We set  $$p(\theta)=2\cos \theta  +2z\arcsin (a\sin \frac \theta 2), \ \theta \in [0,\pi].$$
Using the Maple software we obtain that
$$\min _{\theta\in [0,\pi]}[(p''(\theta))^2+(p'''(\theta))^2]\geq \min\Big\{4+ \frac{{z^2a^2(a^2-1)}^2}{16}, \frac {a^2}{4(1-a^2)}\big(z-\frac {4\sqrt{1-a^2}}a\big)^2\Big\}.$$
If $z$ is such that $|z-\frac {4\sqrt{1-a^2}}a|\geq \eps>0$ then Van der Corput's lemma applied to the phase function $p(\theta)+y\theta/t$ guarantees that
$$\Big| \int _0^\pi e^{itp(\theta)}e^{iy\theta}\psi(\theta) d\theta \Big|\leq C(a,\eps) (|t|+1)^{-1/3}.$$
Assume now that $|z-\frac {4\sqrt{1-a^2}}a|<\eps$ with $\eps$ small enough that we will specify later. Let us write
$$z=\frac {4\sqrt{1-a^2}}a+b$$
with $b$ a small parameter such that $|b|< \eps$. With this notation $p(\theta)=p_b(\theta)=q(\theta)+br(\theta)$ where
$$q(\theta)=2\cos(\theta)+\frac {8\sqrt{1-a^2}}a\arcsin(a\sin \frac \theta 2)$$
and
$$r(\theta)=2\arcsin(a\sin \frac \theta 2).$$
Solving system $(q''(\theta), q'''(\theta))=(0,0)$ with Maple software we obtain that it has a unique solution $\theta=\pi$. Thus for any $\delta<\pi$  there exists a positive constant $c(a,\delta)$ such that
$$|q''(\theta)|+|q'''(\theta)|\geq c(a,\delta),\quad \forall\ \theta\in [0,\pi-\delta].$$
It implies the existence of an $\epsilon=\eps(a, \delta)$ such that for all $|b|\leq \eps$
$$|p_b''(\theta)|+|p_b'''(\theta)|\geq c(a, \delta)-|b|\sup_{x\in [0,\pi]}(|r''|+|r'''|)\geq \frac{c(a,\delta)}2, \quad \forall\ \theta\in [0,\pi-\delta].$$
Hence, Van der Corput's Lemma applied to the phase function $p_b(\theta)+y\theta/t$ guarantees that
$$\Big|\int _0^{\pi-\delta}e^{itp_b(\theta)}e^{iy\theta}\psi (\theta)d\theta\Big|\leq C(a,\delta)(|t|+1)^{-1/3}, \quad \forall |b|< \eps, \forall \ t,y \in \rr.$$
The proof of \eqref{int.osc.2} is finished.

To prove estimate \eqref{int.osc} it remains to show  that we can choose $\delta(a)$ small enough such that for all $|b|<\eps$
\begin{equation}\label{int.near.pi}
|I_b(t)|:=\Big|\int _{\pi-\delta(a)}^\pi e^{itp_b(\theta)}e^{iy\theta}\sin (\theta)d\theta\Big|\leq C(a )(|t|+1)^{-1/3}, \quad \forall y,t \in \rr.
\end{equation}
The Taylor expansions of $q$ and $r$ near $\theta=\pi$ are as follows
$$q(\theta)={\frac {-2a+8\sqrt {1-{a}^{2}}\arcsin \left( a \right) }{a}}-\frac1{16}
\,{\frac { \left( 2\,{a}^{2}-1 \right)  \left( \theta-\pi  \right) ^{4}}{-1
+{a}^{2}}}-{\frac {1}{384}}{\frac { \left( 4\,{a}^{2}-1 \right)
 \left( \theta-\pi  \right) ^{6}}{ \left( -1+{a}^{2} \right) ^{2}}}
+ O ( (\theta -\pi)^8),
$$
and
$$r(\theta)= 2\,\arcsin \left( a \right) -\frac 14\,{\frac {a}{\sqrt {1-{a}^{2}}}}
 \left( \theta-\pi  \right) ^{2}+{\frac {1}{192}}\,{\frac {a\, \left( 2\,{a
}^{2}+1 \right) }{ \left( 1-{a}^{2} \right) ^{3/2}}} \left( \theta-\pi
 \right) ^{4}+O \left(  \left( \theta-\pi  \right) ^{6} \right).
$$
Also the second derivatives of $q$ and $r$ satisfy
$$q''(\theta)=-\frac3{4}
\,{\frac { \left( 2\,{a}^{2}-1 \right)  \left( \theta-\pi  \right) ^{2}}{-1
+{a}^{2}}}+O(|\theta-\pi|^4) \quad \text{as } \theta \sim \pi,$$
and
$$r''(\theta)=-\frac 12\,{\frac {a}{\sqrt {1-{a}^{2}}}}
+O(\theta-\pi)^2 \quad \text{as } \theta \sim \pi.$$
Observe that for $a\neq 1/\sqrt{2}$,  the second derivative of $q$ behaves as $(\theta-\pi)^2$  near $\theta=\pi$. Otherwise it behaves as $(\theta-\pi)^4$ near the same point. Since the proof of \eqref{int.near.pi} is quite different in the two cases we will treat then separately.

In the sequel $\delta(a)$ is chosen such that we can compare $q$ and $r$ with their Taylor expressions near $\theta=\pi$.

{\textbf{Case 1. $a\neq 1/\sqrt 2$}.} The main idea  is to split the interval $[\pi-\delta(a),\pi]$ in three intervals where we can compare $|\theta-\pi|$  with $|b|^{1/2}$ and decide which of them dominates the other:
$$[\pi-\delta(a),\pi]=[\pi-\delta(a),\pi-\alpha_2|b|^{1/2}]\cup [\pi-\alpha_2|b|^{1/2},\pi-\alpha_1|b|^{1/2}]\cup[\pi-\alpha_1|b|^{1/2},\pi],$$
 where $\alpha_1<<1<<\alpha_2$ are independent of $b$ but depend on the parameter $a$. More precisely the parameters
 $\alpha_1$ and $\alpha_2$ are chosen in terms of the first two coefficients of the  Taylor expansion of functions $q$ and $r$ near
 $\theta=\pi$.

Let us consider the interval $[\pi-\delta(a),\pi-\alpha_2|b|^{1/2}]$ with $\alpha_2$ large enough.
In this interval $|\theta-\pi|$ dominates $|b|^{1/2}$ and we apply Lemma \ref{kpv}. We check the hypotheses of this lemma. In this  interval the first derivative of $p_b$ is of the same order as $|\theta-\pi|^3$ :
$$|p'_b(\theta)|\geq |q'(\theta)|-|b||r'(\theta)|\geq C_1|\theta-\pi|(|\theta-\pi|^2-C_2|b|)\geq C_3|\theta-\pi|^3$$
and
$$|p'_b(\theta)|\leq |q'(\theta)|+|b||r'(\theta)|\geq C_4|\theta-\pi|(|\theta-\pi|^2+C_5|b|)\geq C_6|\theta-\pi|^3.$$
Also, the second derivative satisfies:
$$|p''_b(\theta)|\geq |q''(\theta)|-|b||r''(\theta)|\geq C_7(|\theta-\pi|^2-C_8|b|)\geq C_9|\theta-\pi|^2$$
and
$$|p''_b(\theta)|\leq |q''(\theta)|+|b||r''(\theta)|\geq C_{10}(|\theta-\pi|^2+C_{11}|b|)\geq C_{12}|\theta-\pi|^2.$$
We emphasize that all the above constants are independent of $b$. Observe that on the considered interval $|p_b''|\gtrsim |b|$. If we try to apply Van der Corput's Lemma with $k=2$ we obtain
$$\Big|\int_{\pi-\delta(a)}^{\pi-\alpha_2|b|^{1/2}} e^{itp_b(\theta)}e^{iy\theta}\sin (\theta)d\theta\Big|\leq (|tb|)^{-1/2}\max_{[\pi-\delta(a),\pi-\alpha_2|b|^{1/2}]}|\sin \theta|\leq C(\delta(a))|tb|^{-1/2},
$$
an estimate that is not uniform in the parameter $b$.

However, using  Lemma \ref{kpv} we obtain the existence of a constant $C$ depending on all the constants $C_i, i=1,...,12$ but independent of the parameter $b$,
 such that
\begin{align}\label{case.1.est.1}
\Big|&\int_{\pi-\delta(a)}^{\pi-\alpha_2|b|^{1/2}} e^{itp_b(\theta)}e^{iy\theta}\sin (\theta)d\theta\Big|=
\Big|\int_{\pi-\delta(a)}^{\pi-\alpha_2|b|^{1/2}} e^{itp_b(\theta)}e^{iy\theta}|p_b''(\theta)|^{1/2}\frac{\sin (\theta)}{|p_b''(\theta)|^{1/2}}d\theta\Big|\\
\nonumber &\leq C |t|^{-1/2}\Big(\max_{[{\pi-\delta(a)},{\pi-\alpha_1|b|^{1/2}]}}\frac{|\sin (\theta)|}{|p_b''(\theta)|^{1/2}}+
\int_{\pi-\delta(a)}^{\pi-\alpha_2|b|^{1/2}}\Big |\Big(\frac{\sin (\theta)}{|p_b''(\theta)|^{1/2}}\Big)'(\theta)\Big|d\theta\Big)\\
 \nonumber &  \leq C |t|^{-1/2} \max_{[{\pi-\delta(a)},{\pi-\alpha_2|b|^{1/2}]}}\frac{|\sin (\theta)|}{|p_b''(\theta)|^{1/2}}\lesssim
C |t|^{-1/2} \max_{[{\pi-\delta(a)},{\pi-\alpha_2|b|^{1/2}]}}\frac{|\sin (\theta)|}{|\theta-\pi|}\lesssim C |t|^{-1/2}.
\end{align}

On the interval $[\pi-\alpha_2|b|^{1/2},\pi-\alpha_1|b|^{1/2}]$  the third derivative of $p_b$ satisfies:
$$|p'''(\theta)|\simeq |\theta-\pi||C(a)+b|\simeq |b|^{1/2},$$
since $C(a)\neq 0$ in the case $a\neq 1/\sqrt 2$. Applying Van der Corput's Lemma with $k=3$ we get
\begin{equation}\label{case.1.est.2}
\Big|\int_{\pi-\alpha_2|b|^{1/2}}^{\pi-\alpha_1|b|^{1/2}}e^{itp_b(\theta)}e^{iy\theta}\sin (\theta)d\theta\Big|\lesssim (|t b|^{1/2})^{-1/3}\max_{\theta\in[\pi-\alpha_2|b|^{1/2},\pi-\alpha_1|b|^{1/2}]}{|\sin \theta|}\lesssim |t|^{-1/3}.
\end{equation}

On interval $[\pi-\alpha_1|b|^{1/2},\pi]$ with $\alpha_1$ small enough, the term $|br''(\theta)|$ dominates $|q''(\theta)|$. The the behavior of  $p_b''(\theta)$ is given by $|br''(\theta)|$:
$$|p''_b(\theta)|\geq |br''(\theta)| -|q''(\theta)|\geq C_1(|b|-C_2|\theta-\pi|^2)\geq C_3|b|,$$
for some positive constants $C_1$ and $C_2$ independent of the parameter $b$.
Applying Van der Corput's Lemma with $k=2$ we get
\begin{equation}\label{case.1.est.3}
\Big|\int_{\pi-\alpha_1|b|^{1/2}}^{\pi}e^{itp_b(\theta)}e^{iy\theta}\sin (\theta)d\theta\Big|\lesssim (|t b|)^{-1/2}\max_{\theta\in[\pi-\alpha_1|b|^{1/2},\pi]}{|\sin \theta|}\lesssim |t|^{-1/2}.
\end{equation}

Using \eqref{case.1.est.1}, \eqref{case.1.est.2} and \eqref{case.1.est.3} we obtain that \eqref{int.near.pi} holds uniformly for all $|b|<\eps, y$ and $t$ real numbers.

\medskip
{\textbf{Case 2. $a= 1/\sqrt 2$}.} In this case the Taylor expansion of function $q$  at $\theta=\pi$ is given by
$$q(\theta)={\frac {-2a+8\,\sqrt {1-{a}^{2}}\arcsin \left( a \right) }{a}}-{\frac {1}{384}}\,{\frac { \left( 4\,{a}^{2}-1 \right)
 \left( \theta-\pi  \right) ^{6}}{ \left( -1+{a}^{2} \right) ^{2}}}+O(|\theta-\pi|^8).$$
We split the interval $[\pi-\delta(a),\pi]$ as follows:
\begin{align*}[\pi-\delta(a),\pi]=&[\pi-\delta(a),\pi-\alpha_3|b|^{1/4}]\cup[\pi-\alpha_3|b|^{1/4},\pi-\alpha_2|b|^{1/4}]\\
&\cup[\pi-\alpha_2|b|^{1/4},\pi-\alpha_1|b|^{1/2}]\cup[\pi-\alpha_1|b|^{1/2},\pi],
\end{align*}
where $\alpha_2<<1<<\alpha_3$ and all $\alpha_1, \alpha_2, \alpha_3$ are independent of $b$.

On the first interval $[\pi-\delta(a),\pi-\alpha_3|b|^{1/4}]$ we apply Lemma \ref{kpv.cu3}. We have to check that the first third derivatives behave as powers of $|\theta-\pi|$ in this interval.
Observe that
$$|p_b'(\theta)|\geq C_1|\theta-\pi|(|\theta-\pi|^4-C_2|b|)\geq C_3|\theta-\pi|^5$$
and
$$|p_b'(\theta)|\leq C_4|\theta-\pi|(|\theta-\pi|^4+C_5|b|)\geq C_6|\theta-\pi|^5.$$
In a similar manner
$$C_7|\theta-\pi|^4\leq |p_b''(\theta)|\leq C_8|\theta-\pi|^4.$$
Also the third derivative satisfies
$$|p_b'''(\theta)|\geq C_9|\theta-\pi|(|\theta-\pi|^2-C_{10}|b|)\geq C_{11}|\theta-\pi|^3$$
and
$$|p_b'''(\theta)|\leq C_{12}|\theta-\pi|(|\theta-\pi|^2+C_{13}|b|)\geq C_{14}|\theta-\pi|^3.$$
We now apply Lemma \ref{kpv.cu3} taking into account that all the above constants are independent of $b$ and we obtain
\begin{align}\label{case.2.est.1}
\Big|\int_{\pi-\delta(a)}^{\pi-\alpha_3|b|^{1/4}}&e^{itp_b(\theta)}e^{iy\theta}\sin \theta d\theta\Big|=
\Big|\int_{\pi-\delta(a)}^{\pi-\alpha_3|b|^{1/4}}e^{itp_b(\theta)}e^{iy\theta}|p_b'''(\theta)|^{1/3}\frac{\sin \theta}{|p_b'''(\theta)|^{1/3}}d\theta\Big|\\
\nonumber& \lesssim|t|^{-1/3}\Big(\max_{[\pi-\delta(a),\pi-\alpha_3|b|^{1/4}]}\frac{|\sin \theta|}{|p_b'''(\theta)|^{1/3}} +
\int _{\pi-\delta(a)}^{\pi-\alpha_3|b|^{1/4}} \Big|\Big(\frac{\sin \theta}{|p_b'''(\theta)|^{1/3}} \Big)'\Big|  d\theta\Big)\\
\nonumber &  \lesssim|t|^{-1/3}\max_{[\pi-\delta(a),\pi-\alpha_3|b|^{1/4}]}\frac{|\sin \theta|}{|p_b'''(\theta)|^{1/3}} \\
\nonumber &  \lesssim|t|^{-1/3}\max_{[\pi-\delta(a),\pi-\alpha_3|b|^{1/4}]}\frac{|\sin \theta|}{|\theta-\pi|} \leq C|t|^{-1/3}.
\end{align}

In the case of the interval $[\pi-\alpha_3|b|^{1/4},\pi-\alpha_2|b|^{1/4}]$ we apply Van der Corput's Lemma  with $k=3$ and use that
$$|p'''_b(\theta)|\geq C_1|\theta-\pi|(|\theta-\pi|^2-C_2|b|)\geq  C_1|\theta-\pi|(\alpha_2^2 |b|^{1/2}-C_2|b|)\geq C_3 |b|^{1/4+1/2}.$$
Then
\begin{align}\label{case.2.est.2}
\Big|\int_{\pi-\alpha_3|b|^{1/4}}^{\pi-\alpha_2|b|^{1/4}}&e^{itp_b(\theta)}e^{iy\theta}\sin \theta\Big|\leq (|t||b|^{3/4})^{-1/3}\max_{[\pi-\alpha_3|b|^{1/4},\pi-\alpha_2|b|^{1/4}]}|\sin \theta|\leq C|t|^{-1/3}.
\end{align}

Let us now consider the integral on the interval $[\pi-\alpha_2|b|^{1/4},\pi-\alpha_1|b|^{1/2}].$ Observe that in this case
\begin{align}\label{case.2.est.3.1}
\Big|\int_{\pi-\alpha_2|b|^{1/4}}^{\pi-\alpha_1|b|^{1/2}} & e^{itp_b(\theta)}e^{iy\theta}\sin \theta d\theta\Big|\leq \int_{\pi-\alpha_2|b|^{1/4}}^{\pi-\alpha_1|b|^{1/2}}|\sin \theta|d\theta \leq\int_{\alpha_1|b|^{1/2}}^{\alpha_2|b|^{1/4}}|\sin \theta|d\theta \\
\nonumber&\leq \int_{\alpha_1|b|^{1/2}}^{\alpha_2|b|^{1/4}}\theta d\theta\leq C|b|^{1/2}\leq C|t|^{-1/3},
\end{align}
as long as $|b|\leq |t|^{-2/3}$.

We now consider the case $|b|\geq |t|^{-2/3}$ and prove that a similar estimate can be obtained. Observe that on the considered interval the second derivative of $p_b$ satisfies
$$|p''_b(\theta)|\geq |b||r''(\theta)|-|q''(\theta)|\geq C_1 (|b|-C_2|\theta-\pi|^4)\geq C_1 (|b|-C_2(\alpha_2 |b|^{1/4})^4)\geq C_3|b|.$$
Thus, Van der Corput's Lemma with $k=2$ gives us
\begin{align}\label{case.2.est.3.2}
\Big|\int _{\pi-\alpha_2|b|^{1/4}}^{\pi-\alpha_1|b|^{1/2}}&e^{itp_b}e^{iy\theta}\sin \theta d\theta\Big|\lesssim (|tb|)^{-1/2}\max_{\theta\in [\pi-\alpha_2|b|^{1/4},\pi-\alpha_1|b|^{1/2}]}|\sin \theta|\leq (|tb|)^{-1/2}|b|^{1/4}\\
\nonumber&\leq |t|^{-1/2}|b|^{-1/4}\leq  |t|^{-1/2} |t|^{1/6}=|t|^{-1/3}.
\end{align}

On the last interval $[\pi-\alpha_1|b|^{1/2},\pi]$ the term $|br''(\theta)|$ dominates $|q''(\theta)|$. Then the behavior of  $p_b''(\theta)$ in the considered interval is given by $|br''(\theta)|$:
$$|p''_b(\theta)|\geq |br''(\theta)| -|q''(\theta)|\geq C_1(|b|-C_2|\theta-\pi|^4)\geq C_3|b|.$$
Thus
\begin{equation}\label{case.2.est.4}
\Big|\int_{\pi-\alpha_1|b|^{1/2}}^{\pi}e^{itp_b(\theta)}e^{iy\theta}\sin (\theta)d\theta\Big|\lesssim (|t b|)^{-1/2}\max_{\theta\in[\pi-\alpha_1|b|^{1/2},\pi]}{|\sin \theta|}\lesssim |t|^{-1/2}.
\end{equation}

Using the previous estimates \eqref{case.2.est.1}, \eqref{case.2.est.2}, \eqref{case.2.est.3.1}, \eqref{case.2.est.3.2} and
\eqref{case.2.est.4} we obtain that estimate \eqref{int.near.pi} also holds in the case $a=1/\sqrt 2$.

The proof of Lemma \ref{lemma.forte} is now finished.
\end{proof}

In the case of system \eqref{model2} the proof of Theorem \ref{main.model.2} follows the lines of the proof of Theorem \ref{main} by taking into account the representation formula for the resolvent of the operator $A$  given by
\eqref{matrice.model2}.

\begin{lemma}\label{resolvent.model2}
Let $\lambda \in \cn\setminus [-4\max\{b_1^{-2},b_2^{-2}\} ,0]$ and $A$ given by \eqref{matrice.model2}. For any $g\in l^2(\zz^*)$ there exists a unique solution $f\in l^2(\zz^*)$ of the equation
$(A-\lambda I)f=g$. Moreover, it is given by the following formula
\begin{align}\label{rez.2}
f(j)=\frac {-r_s^{|j|}}{b_1^{-2}(r_1^{-1}-r_1)+b_2^{-2}(r_2^{-1}-r_2)}&\Big[g(0)+\sum _{k\in Z_1}r_1^{|k|}g(k)+\sum _{k\in Z_2}r_2^{|k|}g(k)\Big]\\
\nonumber&+\frac{b_s^2}{r_s-r_s^{-1}}\sum _{k\in Z_s}(r_s^{|j-k|}-r_s^{|j|+|k|})g(k),\quad j\in Z_s,
\end{align}
 where for $s\in \{1,2\}$, $r_s=r_s(\lambda)$ is the unique solution with $|r_s|<1$ of the equation
 $$r_s^2-2r_s+1=\lambda b_s^2 r_s.$$
\end{lemma}

We leave the complete details of the proof of Theorem \ref{main.model.2}  to the reader.

\section{Open problems}
\setcounter{equation}{0}

In this article we have analyzed the dispersive properties of the solutions of a system consisting in coupling two discrete Schr\"odinger equations. However we do not cover the case when more discrete equations are coupled. The main difficulty is to write in an accurate and  clean way the resolvent of the linear operator occurring in the system.
Once this case will be understood then we can treat discrete Sch\"odinger equations on trees similar to those considered in \cite{liv-tree} in the continuous  case.

There is another question which arises from this paper. Suppose that we have a system $iU_t+AU=0$ with an initial datum at $t=0$, where
$A$ is an symmetric operator with a finite number of diagonals not identically vanishing. Under which assumptions on the operator $A$ does solution $U$ decay and how we can characterize the decay property in terms of the properties of $A$? When $A$ is a diagonal operator we can use Fourier's analysis tools but in the case of a non-diagonal operator this is not useful.

\medskip
 {\bf
Acknowledgements.}
The authors have been supported by the  grant "Qualitative properties of the partial differential equations and their numerical approximations"-TE 4/2009
of CNCSIS Romania.
L. I. Ignat has been also supported by the grants  LEA CNRS Franco-Roumain MATH-MODE,
MTM2008-03541 of the Spanish MEC and the Project  PI2010-04 of the Basque Government.
D. Stan has been  also supported by a BitDefender fellowship.

Parts of this paper have been developed during the first author's visit to BCAM - Basque Center for Applied Mathematics under the Visiting Fellow program. The authors thank V. Banic\u a and S. Str\u atil\u a for fruitful discussions.

\bibliographystyle{plain}
\bibliography{biblio1}

\begin{thebibliography}{10}

\bibitem{MR2049025}
Valeria Banica.
\newblock Dispersion and {S}trichartz inequalities for {S}chr\"odinger
  equations with singular coefficients.
\newblock {\em SIAM J. Math. Anal.}, 35(4):868--883 (electronic), 2003.

\bibitem{1055.35003}
T.~Cazenave.
\newblock {\em {Semilinear Schr\"{o}dinger equations.}}
\newblock {Courant Lecture Notes in Mathematics 10. Providence, RI: American
  Mathematical Society (AMS); New York, NY: Courant Institute of Mathematical
  Sciences. xiii }, 2003.

\bibitem{MR801582}
J.~Ginibre and G.~Velo.
\newblock The global {C}auchy problem for the nonlinear {S}chr\"odinger
  equation revisited.
\newblock {\em Ann. Inst. H. Poincar\'e Anal. Non Lin\'eaire}, 2(4):309--327,
  1985.

\bibitem{liv-tree}
L.~I. Ignat.
\newblock Strichartz estimates for the {S}chr\"odinger equation on a tree and
  applications.
\newblock {\em SIAM J. Math. Anal.}, in press.

\bibitem{1063.35016}
L.I. Ignat and E.~Zuazua.
\newblock {Dispersive properties of a viscous numerical scheme for the
  Schr\"{o}dinger equation.}
\newblock {\em C. R. Acad. Sci. Paris, Ser. I}, 340(7):529--534, 2005.

\bibitem{MR2485456}
Liviu~I. Ignat and Enrique Zuazua.
\newblock Numerical dispersive schemes for the nonlinear {S}chr\"odinger
  equation.
\newblock {\em SIAM J. Numer. Anal.}, 47(2):1366--1390, 2009.

\bibitem{0922.35028}
M.~Keel and T.~Tao.
\newblock {Endpoint Strichartz estimates.}
\newblock {\em Am. J. Math.}, 120(5):955--980, 1998.

\bibitem{0738.35022}
C.E. Kenig, G.~Ponce, and L.~Vega.
\newblock {Oscillatory integrals and regularity of dispersive equations}.
\newblock {\em Indiana Univ. Math. J.}, 40(1):33--69, 1991.

\bibitem{MR2282998}
A.~I. Komech, E.~A. Kopylova, and M.~Kunze.
\newblock Dispersive estimates for 1{D} discrete {S}chr\"odinger and
  {K}lein-{G}ordon equations.
\newblock {\em Appl. Anal.}, 85(12):1487--1508, 2006.

\bibitem{MR2468536}
D.~E. Pelinovsky and A.~Stefanov.
\newblock On the spectral theory and dispersive estimates for a discrete
  {S}chr\"odinger equation in one dimension.
\newblock {\em J. Math. Phys.}, 49(11):113501, 17, 2008.

\bibitem{MR2150357}
Atanas Stefanov and Panayotis~G. Kevrekidis.
\newblock Asymptotic behaviour of small solutions for the discrete nonlinear
  {S}chr\"odinger and {K}lein-{G}ordon equations.
\newblock {\em Nonlinearity}, 18(4):1841--1857, 2005.

\bibitem{0821.42001}
E.M. Stein.
\newblock {\em {Harmonic analysis: Real-variable methods, orthogonality, and
  oscillatory integrals.}}
\newblock {Princeton Mathematical Series. 43. Princeton, NJ: Princeton
  University Press }, 1993.

\bibitem{0372.35001}
R.S. Strichartz.
\newblock {Restrictions of Fourier transforms to quadratic surfaces and decay
  of solutions of wave equations.}
\newblock {\em Duke Math. J.}, 44:705--714, 1977.

\bibitem{MR2233925}
T.~Tao.
\newblock {\em Nonlinear dispersive equations}, volume 106 of {\em CBMS
  Regional Conference Series in Mathematics}.
\newblock Published for the Conference Board of the Mathematical Sciences,
  Washington, DC, 2006.
\newblock Local and global analysis.

\end{thebibliography}

\end{document}